\documentclass[hidelinks,onefignum,onetabnum]{siamart250211}

\usepackage[T2A]{fontenc}
\usepackage[cp1251]{inputenc}
\usepackage[english]{babel}
\usepackage{ifpdf}
\usepackage{hyperref}
\usepackage{bm}
\newcommand{\markercross}{\scalebox{0.8}{$\pmb{\times}$}}

\newcommand{\markerline}{\rule[0.12em]{0.75em}{0.12em}}

\hypersetup{
    pagebackref=true,
	colorlinks=true,
	linkcolor=blue,
	filecolor=magenta,      
	urlcolor=cyan,
	pdftitle={Subordination based approximation of Caputo fractional propagator and related numerical methods},
	pdfauthor={Dmytro Sytnyk}
	pdfpagemode=FullScreen,
}
\usepackage[capitalize,nameinlink]{cleveref}
\usepackage{latexsym,amsfonts,amssymb,amsmath}
\usepackage{graphics,graphicx,longtable,array,euscript,epsfig}
\graphicspath{{./pic/}}

\usepackage[sort&compress,numbers]{natbib}
\ifpdf
\usepackage{orcidlink} 
\usepackage{overpic}
\usepackage[hyperpageref]{backref} 
\renewcommand*{\backrefalt}[4]{%
	\hypersetup{linkcolor=gray}%
	\color{gray}{%
		[%
		\ifcase #1 %
		No citations%
		\or
		Cited on p. #2%
		\else
		Cited on pp. #2%
		\fi
		]%
	}
}
\usepackage[author={Dmytro},draft]{fixme}
\else

\fi
\usepackage{enumitem} 
\crefformat{enumi}{(#2#1#3)} 

\usepackage[section]{placeins}

\ifpdf
\newsiamremark{remark}{Remark} 
{
	\theoremstyle{plain}
	\theoremheaderfont{\normalfont\itshape}
	\theorembodyfont{\normalfont}
	\theoremseparator{.}
	\theoremsymbol{}
	\newtheorem{example}{Example} 
}
\else 
	\newtheorem{remark}{Remark} 
	\newtheorem{example}{Example} 
\fi
\crefname{remark}{Remark}{Remark} 
\crefname{Example}{Example}{Example} 
\crefname{subsection}{Section}{Section} 
\crefname{section}{Section}{Section} 

\newcommand{\R}{\mathbb R}
\newcommand{\oC}{\mathbb C}
\newcommand{\N}{\mathbb N}

\newcommand{\cF}{\mathcal{F}}

\newcommand{\cO}{\mathcal{O}}

\DeclareMathOperator{\Arg}{Arg}

\DeclareMathOperator*{\Res}{Res}

\newcommand{\suml}{\sum\limits}
\newcommand{\intl}{\int\limits}

\newcommand{\e}{\mathrm{e}}

\newcommand{\wt}[1]{\widetilde{#1}}

\allowdisplaybreaks

\begin{document}
\title{Subordination based approximation of Caputo fractional propagator and related numerical methods.}

\ifpdf
	\author{Dmytro Sytnyk\,\orcidlink{0000-0003-3065-4921}%
		\thanks{%
			Department of Numerical Mathematics, Institute of Mathematics, National Academy of Sciences of Ukraine, Kyiv, 01024, Ukraine; (\email{sytnik@imath.kiev.ua}). }%
	}
\else
	\author{\href{https://orcid.org/0000-0003-3065-4921}{Dmytro Sytnyk}%
	\thanks{%
		Department of Numerical Mathematics, Institute of Mathematics, National Academy of Sciences of Ukraine, Kyiv, 01024, Ukraine; (\href{mailto:sytnik@imath.kiev.ua}{sytnik@imath.kiev.ua}). }%
	}
\fi
\date{\today}
\maketitle

\begin{abstract}
	In this work, we propose an exponentially convergent numerical method for the Caputo fractional propagator $S_\alpha(t)$ and the associated mild solution of the Cauchy problem with time-independent sectorial operator coefficient $A$ and Caputo fractional derivative of order $\alpha \in (0,2)$ in time. The proposed methods are constructed by generalizing the earlier developed approximation of $S_\alpha(t)$ with help of the subordination principle. Such technique permits us to eliminate the dependence of the main part of error estimate on $\alpha$, while preserving other computationally relevant properties of the original approximation: native support for multilevel parallelism, the ability to handle initial data with minimal spatial smoothness, and stable exponential convergence for all $t \in [0, T]$. Ultimately, the use of subordination leads to a significant improvement of the method's convergence behavior, particularly for small $\alpha < 0.5$, and opens up further opportunities for efficient data reuse. To validate theoretical results, we consider applications of the developed methods to the direct problem of solution approximation, as well as to the inverse problem of fractional order identification.
\end{abstract}


\ifpdf
	\begin{AMS}
		{34A08, 34K37, 35R11, 35R20,  65L05, 65J08, 65J10, 65M32}
	\end{AMS}
\else
\fi
\section{Introduction}\label{FCPML_Intro}
In this work we consider solution operators associated with the following fractional-in-time Cauchy problem
\begin{subequations}\label{eq:FCP_DEBC}
	\begin{eqnarray}
		&\hspace*{-1.6em}\partial_t^\alpha u(t) + A u(t) = f(t),\quad  t  \in (0, T], \quad \alpha \in (0, 2),
		\label{eq:FCP_DE}\\
		&\begin{cases}
			u(0) = u_0,                & \text{if $0 < \alpha \leq 1$}, \\
			u(0) = u_0, \ u'(0) = u_1, & \text{if $1< \alpha < 2$}.
		\end{cases}
		\label{eq:FCP_BC}
	\end{eqnarray}
\end{subequations}
Here,  $\partial_t^\alpha$ denotes the Caputo fractional derivative of order
$\alpha$
(see, e.g. \cite[p.\,91]{Kilbas2006})
\begin{equation}\label{eq:FracDeriv_Caputo}
	\partial_t^\alpha u(t)= \frac{1}{\Gamma(n-\alpha)} \intl_0^t(t-s)^{n-\alpha-1}u^{(n)}(s)\, ds,
\end{equation}
with $u^{(n)}(s)$ being the standard derivative of integer order $n = \lceil \alpha \rceil$.
For clarity, we enforce $\partial_t^1 u(t) = u'(t)$.

The time-independent coefficient $A$ in \eqref{eq:FCP_DE}  is assumed to be a closed linear operator with the domain $D(A)$ dense in a Banach space $X = X(\|\cdot\|,\Omega)$ 
and the spectrum $\mathrm{Sp}(A)$ contained in the sectorial region $\Sigma(\rho_s, \varphi_s)$ of complex plane:
\begin{equation}\label{eq:SpSector}
	\Sigma(\rho_s, \varphi_s) = \left\{ z=\rho_s+\rho \e^{i\theta}:\quad \rho \in [0,\infty), \ \left|\theta\right| < \varphi_s
	\right\}, \quad \rho_s > 0, \  \varphi_s< \frac{\pi}{2}.
\end{equation}
The numbers $\rho_s$ and $\varphi_s$ are called spectral parameters of $A$.
In addition to the assumptions on the spectrum, we suppose that the resolvent of $A$:
$
	R\left(z, A\right) \equiv (zI-A)^{-1}
$
satisfies the bound
\begin{equation}\label{eq:ResSector}
	\left\|(zI-A)^{-1} x\right\| \leq \frac{M}{1+\left|z\right|} \|x\|, \quad M > 0,
\end{equation}
for all $x \in X$ and any fixed $z \in \oC \setminus \Sigma(\rho_s, \varphi_s)$.
Following the established convention \cite{bGavrilyuk2011}, we will cal such operators $A$ strongly positive.
The class of strongly positive operators includes second order elliptic partial differential operators \cite{Haase2006}, as well as more general strongly elliptic pseudo-differential operators defined over the bounded domain \cite{Bilyj2010}.
Spectral parameters $\rho_s$ and $\varphi_s$ of $A$ can be estimated from the coefficients of its  differential expression \cite{Gavrilyuk2004}.

In this work, we focus on the mild solution to the given problem.
Let us consider a Volterra integral equation
\begin{equation}\label{eq:FCP_Mild}
	u(t) = \suml_{k=0}^{n-1}u_k t^k - J_\alpha A u(t) + J_\alpha f(t),
\end{equation}
where, $J_\alpha$ denotes the Riemann--Liouville integral operator,
\begin{equation*}
	J_\alpha v(t) = \frac{1}{\Gamma(\alpha)} \intl_0^t(t-s)^{\alpha-1}v(s)\, ds.
\end{equation*}
\begin{definition}\label{def:FCP_Mild_Sol}
	Let $\alpha \in (0,2)$, $u_0, u_1 \in X$ and $f \in C([0, T],X)$, a function $u \in C([0, T],D(A))$ is called a mild solution of \eqref{eq:FCP_DEBC} if it satisfies integral equation \eqref{eq:FCP_Mild}.
\end{definition}
If the parameters $\alpha$, $u_0$, $u_1$, $f(t)$ satisfy \cref{def:FCP_Mild_Sol},  the linear Volterra integral equation \eqref{eq:FCP_Mild} admits a unique solution  \cite[Prop. 1.2]{Pruess2013}:
\[
	u(t) = \frac{\mathrm{d}}{\mathrm{dt}}\intl_{0}^tS_\alpha(t - s) \left(u_0 + u_1 s + J_\alpha f(s)\right) \, ds.
\]
Henceforth, we assume that $u_1 \equiv 0$, if $\alpha \in (0, 1]$.
Depending on the available temporal regularity of $f(t)$, this variation of parameters formula may take several different forms \cite{Hernandez2010,Zhou2013,Keyantuo2013,Li2015}.
Comparison of existing representations were conducted in \cite{Sytnyk2023b},
where it was ultimately concluded that the following form of the mild solution to \eqref{eq:FCP_DEBC} is more favorable from the numerical point of view than its analogues
\begin{equation}\label{eq:FCP_InhomSol_rep}
	u(t) = S_\alpha(t) u_0 + \intl_{0}^tS_\alpha(t - s) u_1\, ds + J_\alpha S_{\alpha}(t)f(0) + \intl_0^t S_{\alpha}(t - s)   J_\alpha f'(s) d s.
\end{equation}
The inhomogeneous part of \eqref{eq:FCP_InhomSol_rep} is well-defined in the strong sense for arbitrary finite $t \geq 0$ under assumptions $\exists \chi >0$: $f(0),f'(t) \in D(A^\chi)$, which can be leveraged to construct solution approximations free of the accuracy degradation near $t=0$. (for further details see \cite{Sytnyk2023}).
Such distinctive feature of \eqref{eq:FCP_InhomSol_rep} stems from the fact that this formula,  unlike other solution representations, involves only the problem's native propagator $S_\alpha(t)$.
It is formally defined as follows.
\begin{definition}[\cite{Pruess2013}]\label{def:FCP_Sol_Op}
	Let $\alpha \in (0,2)$. A bounded linear operator $S_\alpha(t): X \rightarrow X$ is called the propagator of \eqref{eq:FCP_DEBC}  if the following three conditions are satisfied:
	\begin{enumerate}[label=P\arabic*., ref=P\arabic*] \itemsep.1em
		\item $S_\alpha(t)$ is strongly continuous on $X$, for $t \geq 0$, and $S_\alpha(0) = I$.%
		      \label{def:FCP_Sol_Op_Id}
		\item $S_\alpha(t)D(A) \subset D(A)$ and $AS_\alpha(t) x = S_\alpha(t)Ax$, for all $x \in D(A)$, $t \geq 0$.%
		      \label{def:FCP_Sol_Op_Comm}
		\item For any $u_0 \in D(A)$, the function $u(t) = S_\alpha(t) u_0$, $t \geq 0$, is the solution of equation  \eqref{eq:FCP_Mild} with $u_1 = 0$ and $f(t) \equiv 0$.%
		      \label{def:FCP_Sol_Op_S}
	\end{enumerate}
\end{definition}
Operator $A$ from the above definition is usually called the generator of $S_{\alpha}(t)$.
For the class of strongly positive generators, $S_{\alpha}(t)$ can be defined via the contour integration.
Let $\Gamma_I$ denote a contour chosen in such a way that $e^{zt}$ remains bounded for $z \in \Gamma_I$ and the curve $z^\alpha$, $z \in \Gamma_I$, is positively oriented with respect to $\mathrm{Sp}(-A)\cup \{0\}$.
Then, the linear operator $S_\alpha(t)$, defined by
\begin{equation}\label{eq:FCP_SO_cont_repr}
	S_{\alpha}(t) x  = \frac{1}{2\pi i} \intl_{\Gamma_I}  e^{zt}z^{\alpha-1} (z^\alpha I + A)^{-1} x dz, \quad x \in X,
\end{equation}
is the propagator of \eqref{eq:FCP_DEBC} \cite[Lem. 1]{Sytnyk2023b}.
Formula \eqref{eq:FCP_SO_cont_repr} offers a convenient way to numerically evaluate mild solution \eqref{eq:FCP_InhomSol_rep} by approximating the involved integral operators with quadratures.
This technique was applied in \cite{Sheen2003,gm5,McLTh}  to the particular variants of \eqref{eq:FCP_DEBC} and was recently used in \cite{Sytnyk2023} to construct the numerical solution of  the given Cauchy problem in its full generality.
We direct an interested reader to \cite{bGavrilyuk2011,Sytnyk2023,Sytnyk2023b} for more details and comparison between the existing methods.

In the current work we focus on a more general contour representation of the fractional propagator, obtained from the identity
\begin{equation}\label{eq:FCP_SO_subord_repr}
	S_\alpha(t) x = \intl_0^\infty \Phi_\gamma(s) S_{\beta}(st^{\gamma}) x\, d s, \quad \gamma = \frac{\alpha}{\beta},
\end{equation}
where $0 < \alpha \leq \beta < 2$ and $\Phi_\gamma(z) = \suml_{n=0}^\infty \frac{(-z)^n}{n! \Gamma(1-\gamma (n + 1) )}$ is the \textit{M}-Wright function \cite{Mainardi2010}.
Formula \eqref{eq:FCP_SO_subord_repr} is a consequence of the so-called subordination principle.
It was substantiated by Pr\"uss \cite{Pruess2013} in a broader context of abstract integral equations with convolution kernels and subsequently applied to equation \eqref{eq:FCP_Mild} by Bazhlekova, who derived the identity \eqref{eq:FCP_SO_subord_repr} in the closed form \cite{Bazhlekova2000,Bazhlekova2001}.
This identity becomes particularly useful if $\beta = 1$,  then the properties of the fractional propagator $S_\alpha(t)$ can be studied using the well-established theory of exponential semigroups $S_1(t) = \exp(-At)$.
The exponential subordination was used in \cite{Bazhlekova2001,Li2010,Chen2016a} to prove the existence and boundedness of $S_\alpha(t)$, $\alpha \in (0,1)$, along with its powers,
and to establish time and space regularity estimates for $A^{\kappa}S_\alpha(t)$, $\kappa > 0$.
Further theoretical developments include extensions of representation \eqref{eq:FCP_SO_subord_repr} to a wider class of so-called almost sectorial operators, and propagators connected with a more general notion of fractional derivative \cite{Bazhlekova2018}; as well as the applications of \eqref{eq:FCP_SO_subord_repr} to nonlinear generalizations of the given fractional Cauchy problem (see \cite{Zhou2016,Lizama2019} and the references therein).

We are interested in exploring the numerical potential of \eqref{eq:FCP_SO_subord_repr}. From the outset this formula does not seem to be computationally appealing. The convergent approximation of \eqref{eq:FCP_SO_subord_repr} requires accurate numerical evaluation of $S_{\beta}(st^{\beta}) x$ on a non-uniform grid, containing both arbitrarily small and large values of $s \in (0, \infty)$. The accuracy of existing time-stepping methods \cite{Garrappa2015a,Baffet2017,jin2019numerical,Khristenko2021,Banjai2022} degrade near $s=0$ due to the singularity of $S_{\beta}(0) x$. The method from our recent work \cite{Sytnyk2023} is free from such degradation; however, it struggles with the evaluation of $S_{\beta}(s) x$ for large $s$. As far as we know, the numerical potential of \eqref{eq:FCP_SO_subord_repr} has been largely left unexplored, with exception of \cite{Aceto2022} and the original works \cite{Bazhlekova2015,Bazhlekova2018} where it was applied to the particular cases of \eqref{eq:FCP_DEBC}.

To circumvent the problem with evaluating $S_{\beta}(s) x$ for a wide range of $s$, in \cref{sec:FCPML_subordinated_approximation}, we combine identity \eqref{eq:FCP_SO_subord_repr} with the contour representation of $S_{\beta}(st^{\beta}) x$ given by \eqref{eq:FCP_SO_cont_repr}.
After rearrangement, it permits us to decouple the evaluation of the time-dependent component in  \eqref{eq:FCP_SO_subord_repr} from the part that depends on $A$, and discretize the resulting representation by the sinc-quadrature rule.
Full description of the discretization procedure, along with its error analysis, is provided in \cref{sec:FCPML_Prop_Approx}.
The constructed approximation of $S_\alpha(t)$ inherits all computationally relevant properties of the method from \cite[Sec. 3.2]{Sytnyk2023}, including the ability to handle arguments with minimal spatial smoothness, capacity for multi-level parallelism, as well as the exponential convergence for any $t \in [0, T]$ in the presence of additional errors caused by the spatial discretization and the finite-precision arithmetic.
Moreover, the spatial component of the proposed evaluation procedure for $S_\alpha(t)x$ relies only on the solutions $r_k$ of the sequence of resolvent equations $(z_k^\beta+A)r_k = x$, making the accuracy of the resulting numerical method for \eqref{eq:FCP_DEBC} independent of $\alpha \in (0, \beta]$.
This offers a measurably faster and more stable error decay compared to the behavior of similar methods from \cite{Sytnyk2023,Colbrook2022a,Fischer2019,McLean2010}, especially for small $\alpha < 0.5$.
In \Cref{sec:FCPML_Applications_Examples}, we move on to apply the proposed propagator approximation to the numerical evaluation of mild solution \eqref{eq:FCP_InhomSol_rep}.
To this end, we adapt relevant results about the discretization of $J_\alpha$, as well as the other components of  \eqref{eq:FCP_InhomSol_rep}, from \cite{Sytnyk2023}.
Here, we also experimentally validate the convergence properties of the developed numerical method, in realistic situations when the spatial discretization of the solution is performed by the sub-exponentially convergent numerical methods, and demonstrate how the mentioned independence of $r_k$ on $\alpha$ and $t$ can lead to considerable performance gains using the fractional-order identification problem as a model example.

\section{Subordination based propagator representation}\label{sec:FCPML_subordinated_approximation}
Our primary interest in \eqref{eq:FCP_SO_subord_repr} is motivated by the numerical potential of this formula for evaluating $S_\alpha(t)$.
More precisely, when $\alpha \leq \beta \leq 1$, the propagator approximation via \eqref{eq:FCP_SO_cont_repr}, \eqref{eq:FCP_SO_subord_repr} should lead to a faster numerical method than the quadrature of \eqref{eq:FCP_SO_cont_repr} alone, because the quadrature's convergence order peaks at $\alpha = 1$.

Let $E_{\alpha,\beta}(z)$ denote a two-parameter Mittag-Leffler function
\begin{equation}
	E_{\alpha,\beta}(z)
	= \suml_{k=0}^{\infty} \frac{z^k}{\Gamma(\alpha n + \beta)}.
\end{equation}
For fixed $\alpha \in (0,1)$, $\beta \in [1, 2]$, this function satisfies the estimate \cite[Thm. 4.3]{Gorenflo2020}
\begin{equation}\label{eq:ML1_bound}
	|E_{\alpha,\beta}(z)| \leq \frac{M_1}{1+|z|}
	+M_2 
	\begin{cases}
		|z|^{\frac{1-\beta}{\alpha}}e^{\Re{z^{1/\alpha}}},&  |\arg{z}| \leq \pi \alpha, \\
		0, & \hspace*{-2.54em}\pi \alpha < |\arg{z}|  < \pi;
	\end{cases}
\end{equation}
with some bounded constants $M_1, M_2 \geq 0$ independent of $z \in \oC$.
Besides, we will also need an identity connecting the Mittag-Leffler and M-Wright functions \cite[Eq. (A.5)]{Bazhlekova2000}:
\begin{equation}\label{eq:Wright_Laplace_transform}
	E_{\gamma,1}(z) = \intl_0^\infty e^{zt} \Phi_\gamma(t) \, dt, \quad z \in \oC, \ \gamma \in (0,1).
\end{equation}

\begin{lemma}\label{lem:FCP_SO_ML_repr}
	Assume that $A$ is a strongly positive operator with spectral parameters $\rho_s, \varphi_s$ and $\beta \in (0, 2)$.
	Then, for any $\alpha \leq \min{\left\{\beta, 2 \left (1 - \frac{\varphi_s}{\pi} \right) \right\} }$, $x \in X$, the operator function $S_{\alpha}(t)$:
	\begin{equation}\label{eq:FCP_SO_MLcont_repr}
		S_{\alpha}(t) x  =
		\frac{1}{2\pi i} \intl_{\Gamma_I} E_{\gamma,1}{\left(z t^{\gamma}\right)}  z^{\beta-1} (z^\beta I + A)^{-1} x\, dz, \quad \gamma = \frac{\alpha}{\beta},
	\end{equation}
	is well defined and bounded.
	Moreover, $S_\alpha(t)$ is the propagator of \eqref{eq:FCP_DEBC}.
	Here, $\Gamma_I$ denotes the contour chosen in such a way that $E_{\gamma,1}{\left(z\right )}$ is uniformly bounded for $z \in \Gamma_I$ and the curve $z^\beta$ is positively oriented with respect to $\mathrm{Sp}(-A)\cup \{0\}$.
\end{lemma}
\begin{proof}
	To begin with, we would like to determine the relations between $\alpha$ and the spectral parameters of $A$ that are sufficient for the existence of $\Gamma_I$.
	According to \eqref{eq:ML1_bound}, the scalar part of the integrand in \eqref{eq:FCP_SO_MLcont_repr} is bounded if $z \in \oC \setminus \{0\}$ and $\Arg{(z)}  \in \left[\frac{\pi \alpha}{2 \beta}, \pi \right]$.
	Meanwhile, the operator part of the integrand is bounded if and only if $\Arg{(z)}  \in \left(0, \frac{\pi-\varphi_s}{\beta}\right)$.
	By combining the conditions that guaranty a nonempty intersection of these two intervals for $\Arg{(z)}$ with the inequality $\alpha \leq \beta$ we get the lemma's constraint on $\alpha$.
	Together with bounds \eqref{eq:ResSector} and \eqref{eq:ML1_bound}, they
	imply the existence of a suitable contour $\Gamma_I$ and the strong convergence of the integral in \eqref{eq:FCP_SO_MLcont_repr} for positive $t$.
	Next, using equality \eqref{eq:Wright_Laplace_transform} and propagator representation \eqref{eq:FCP_SO_cont_repr}, we transform the right-hand side of \eqref{eq:FCP_SO_MLcont_repr} as follows
	\[
		\begin{aligned}
			\frac{1}{2\pi i} \intl_{\Gamma_I}
			 & E_{\gamma,1}{\left(z t^{\gamma}\right)} z^{\beta-1} (z^\beta I + A)^{-1} x\, dz                                                                      \\
			 & = \frac{1}{2\pi i} \intl_{\Gamma_I} \intl_0^\infty \Phi_\gamma(s) \exp{\left(z st^{\gamma}\right)} \, ds \,  z^{\beta-1} (z^\beta I + A)^{-1} x\, dz \\
			 & = \frac{1}{2\pi i} \intl_0^\infty \Phi_\gamma(s) \intl_{\Gamma_I} \exp{\left(z st^{\gamma}\right)} z^{\beta-1} (z^\beta I + A)^{-1} x\, dz ds
			= \intl_0^\infty \Phi_\gamma(s) S_{\beta}(st^{\gamma}) x\, d s.
		\end{aligned}
	\]
	The derived identity demonstrates equivalence between \eqref{eq:FCP_SO_MLcont_repr}  and \eqref{eq:FCP_SO_subord_repr} for all $t>0$.
	When $t=0$, formula \eqref{eq:FCP_SO_MLcont_repr} coincides with propagator representation \eqref{eq:FCP_SO_cont_repr}, moreover $S_\beta(0)  = S_\alpha(0) = I$.
	This concludes the proof.
\end{proof}
We highlight that, in contrast to \eqref{eq:FCP_SO_subord_repr}, formula \eqref{eq:FCP_SO_MLcont_repr} does not require evaluation of the auxiliary propagator $S_\beta(t)$, $t \in (0, \infty)$.
This removes the potential source of numerical instabilities, mentioned in \cref{FCPML_Intro}, and is also deemed more reasonable from a causality perspective \cite{Kochubei2011}.

Before moving to a more in-depth discussion on how the mentioned structural properties of \eqref{eq:FCP_SO_MLcont_repr} affect the quadrature-based approximation of $S_\alpha(t)$, we first need to make sure that the involved integral is convergent in the strong sense at $t = 0$.
To achieve that we will construct a modification of \eqref{eq:FCP_SO_MLcont_repr} using the technique devised in \cite{gm5}.
It relies on the following technical result.
\begin{proposition}[\cite{McLean2010,Sytnyk2023}]\label{prop:FCP_res_cor}
	Let  $A$ be a strongly positive operator.
	If $x \in D(A^{\kappa})$, with some $\kappa \in (0,1]$,  then
	for any $z^\beta \notin \mathrm{Sp}(-A) \cup \{0\}$,
	\begin{equation}\label{eq:FCP_res_cor_norm_est}
		\left\| z^{\beta-1}(z^\beta I+A)^{-1} x - \frac{1}{z} x\right \|
		\leq \frac{K(1+M) \left \|A^{\kappa} x \right \|}{|z|(1+|z|^\beta)^\kappa}, \quad \beta \geq 0,
	\end{equation}
	where $K > 0$ is some constant and $M$ is defined by \eqref{eq:ResSector}.
\end{proposition}
\begin{corollary}\label{thm:FCP_SO_MLcor_cont_repr}
	Assume that $A$, $\alpha$, $\beta$, $\Gamma_I$ satisfy the conditions of \cref{lem:FCP_SO_ML_repr}.
	If $x \in D(A^\kappa)$, for some $\kappa>0$, then the action of propagator $S_\alpha(t)$ on $x$ can be evaluated as follows
	\begin{equation}\label{eq:FCP_SO_MLcor_cont_repr}
		S_\alpha(t)x
		= \frac{1}{2\pi i} \intl_{\Gamma_I} E_{\gamma,1}{\left(z t^{\gamma}\right)} \left( z^{\beta-1} (z^\beta I + A)^{-1}   - \frac{1}{z}I  \right)x\, dz + x, \quad \gamma = \frac{\alpha}{\beta}.
	\end{equation}
	The contour integral in the above formula is strongly convergent for any $t \geq 0$.
\end{corollary}
\begin{proof}
	Owing to the conditions imposed on $ \Gamma_I$ in \cref{lem:FCP_SO_ML_repr} and \eqref{eq:ML1_bound}, the function $zE_{\gamma,1}{\left(z\right)}$, $z \in \Gamma_I$, remains bounded.
	Thus, we can rewrite \eqref{eq:FCP_SO_cont_repr} as
	\[
		S_{\alpha}(t) x
		= \frac{1}{2\pi i} \intl_{\Gamma_I} E_{\gamma,1}{\left(z t^{\gamma}\right)} \left( z^{\beta-1} (z^\beta I + A)^{-1} - \frac{1}{z}I \right) x\, dz + x \intl_{\Gamma_I} \frac{1}{2\pi i z} E_{\gamma,1}{\left(z t^{\gamma}\right)}\, dz.
	\]
	The first integral is convergent in the strong sense for any $t \geq 0$, due to the estimate from \cref{prop:FCP_res_cor} and the assumed regularity of $x$.
	The second integral is equal to $\Res\limits_{z=0} E_{\gamma,1}{\left(z t^{\gamma}\right)}/z
		= 1$,
	which leads us to \eqref{eq:FCP_SO_MLcor_cont_repr}.
\end{proof}
Subordination-based formulas \eqref{eq:FCP_SO_MLcont_repr} and \eqref{eq:FCP_SO_MLcor_cont_repr} generalize contour representation \eqref{eq:FCP_SO_cont_repr} of $S_\alpha(t)$ beyond the case $\alpha=\beta$, without further narrowing the class of suitable operators $A$, described by the inequality $\varphi_s < \pi\min{\left\{\frac{1}{2}, \left(1 - \tfrac{\alpha}{2}\right)\right\}}$.

Moreover, the ability to vary the parameter $\beta$ in \eqref{eq:FCP_SO_MLcor_cont_repr} independently of $\alpha$ permits us to maximize the "angular size" of the region $\Theta \subset \oC$ in which the integrand is decaying super-linearly, when $z \to \infty$.
By "angular size" we refer to a limit of the argument variation within $\Theta$:
$
	\omega(\Theta) = \lim\limits_{r \to +\infty} \mathrm{diam}\left\{\Arg{(z)}: z \in \Theta \land \Re{z} \geq 0 \land |z| = r \right\}.
$
This characteristic of $\Theta$ is crucial for numerical analysis, as it is known to have a direct impact on the convergence speed of the resulting propagator approximation \cite{McLTh, Weideman2010,gm5,Sytnyk2023}.
From the theoretical point of view, explored in the proof of \cref{lem:FCP_SO_ML_repr}, any such  $\Theta$ is contained within $\Theta_m \equiv \Sigma\left(0, \phi_s\right)\setminus \Sigma\left(a_m, \phi_c\right)$, where $\phi_s = \min{\left\{\pi, {(\pi - \varphi_s)}/{\beta}\right\}}$, $\phi_c = {\pi \gamma}/{2}$ and $a_m>0$ is a certain $\Gamma_I$ -- dependent constant (see \cref{fig:FCPML_hyp_cont} (b)).
Whence, for the maximal possible angular size we get $\omega_m = \phi_s -  \phi_c$.

Additional practical considerations may lead to a further reduction of $\Theta$
and enforce $\omega(\Theta) < \omega_m$.
These include the ability to numerically evaluate
the components of \eqref{eq:FCP_SO_MLcor_cont_repr} to the required precision, discussed in \cref{rem:FCPML_OPC_omega}, or the necessity to efficiently compute $S_\alpha(t)$ for multiple values of $\alpha \in (0, \beta]$.
In the sequel, we also consider a more restricted region $\Theta_\star$, which is conformant with the corresponding region $\Theta_1$ for $S_1(t)$:
\[
	\Theta_\star = \Sigma\left(0, \phi_s\right)
	\setminus \Sigma\left(a_m, \tfrac{\pi}{2} \max{\left\{1, \tfrac{1}{\beta}\right\} }\right),
	\quad \omega_\star = \phi_s -  \tfrac{\pi}{2} \max{\left\{1, \tfrac{1}{\beta}\right\}}
\]
The inclusion $\Theta_\star \subseteq \Theta_1$ makes the sub-optimal choice $\Theta = \Theta_\star$  compatible with the commonly used resolvent evaluation techniques such as finite-element and finite-difference (FD) methods \cite{McLTh,Sytnyk2023}.
The approximation of $S_{\alpha}(t)$, discussed below, should also be compatible with other numerical methods for stationary problems, so long as the property $\mathrm{Sp}(-A) \subseteq \Sigma\left(0, \phi_s\right)$ is preserved upon discretization.

More importantly, the setting $\Theta = \Theta_\star$ allows us to entirely remove the dependence of both the resolvent and the contour in \eqref{eq:FCP_SO_MLcor_cont_repr} on the fractional order $\alpha$.
For all practical purposes it means that, after the first numerical evaluation of the propagator $S_\alpha(t)$ for some fixed $\alpha \in (0, \beta]$ and $t \in [0,T]$, the user should be able to reevaluate $S_\alpha(t)$ for any other admissible pairs $\alpha, t$ without the need to recalculate the resolvents of $A$.
The impact of this feature on a computational performance of the developed method is showcased in \cref{ex:FCPML_ex2_inv_alpha_hom_R_FD}, where we deal with identification of $\alpha$ from the set of solution measurements.

\section{Propagator discretization and quadrature}\label{sec:FCPML_Prop_Approx}
In this section, we describe how to numerically evaluate \eqref{eq:FCP_SO_MLcor_cont_repr}
via the sinc-quadrature rule on the contour $\Gamma_I$.
The pursued strategy uses the ideas from paper \cite{gm5}, devoted to the approximation of $S_1(t)$, and the results from author's previous work \cite{Sytnyk2023}, which deals with fractional propagator representation \eqref{eq:FCP_SO_cont_repr}.

We define contour $\Gamma_I$ to be an oriented hyperbolic curve
\begin{equation}\label{eq:int_cont_hyp_crit_sec}
	\Gamma_I: z(\xi) = a_0 - a_I \cosh(\xi) + i b_I \sinh(\xi), \quad \xi \in	(-\infty, \infty ),
\end{equation}
and then select the positive parameters $a_0$, $a_I$, $b_I$, in such a way that the function $z(\xi)$ conformally maps the horizontal strip $D_d$ of the complex plane:
\[
	D_d=\left \{z \in \mathbb{C}: - \infty < \Re z < \infty, |\Im z|<d \right \}.
\]
into the region $z(D_d) \subseteq \Theta$ and $\omega(z(D_d)) = \omega(\Theta)$ (see  \cref{fig:FCPML_hyp_cont}).
\begin{figure}[h!tb]
	\begin{centering}
		\ifpdf
			\begin{overpic}[width=0.7\textwidth,permil=true,viewport=110 200 550 364,clip=true]
				{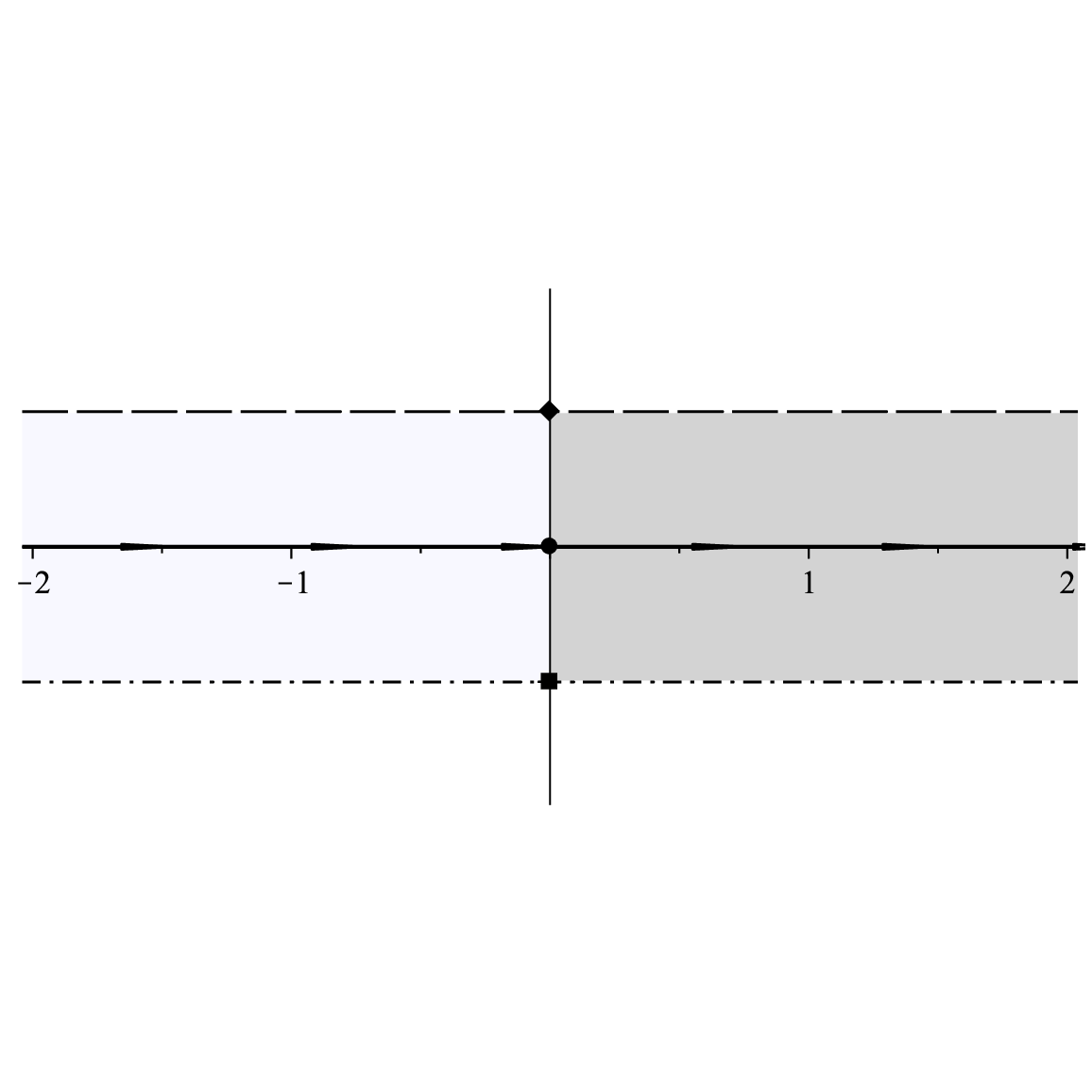}
				\put(350,290){$d$}
				\put(320,40){$-d$}
				\put(1000,190){$\xi$}
				\put(400,340){$\nu$}
				\put(0,350){\scriptsize ({\bf a})}
			\end{overpic}%
		\else
			\includegraphics[width=0.7\textwidth,permil=true,viewport=110 200 550 364,clip=true]%
			{FCP_strip_sphi_pi_6_sa_0_alpha_0_1_ac_pi_2_phic_pi_2}
		\fi
		\\[6pt]
		\ifpdf
			\begin{overpic}[width=0.7\textwidth,permil=true,viewport=61 200 471 480,clip=true]
				{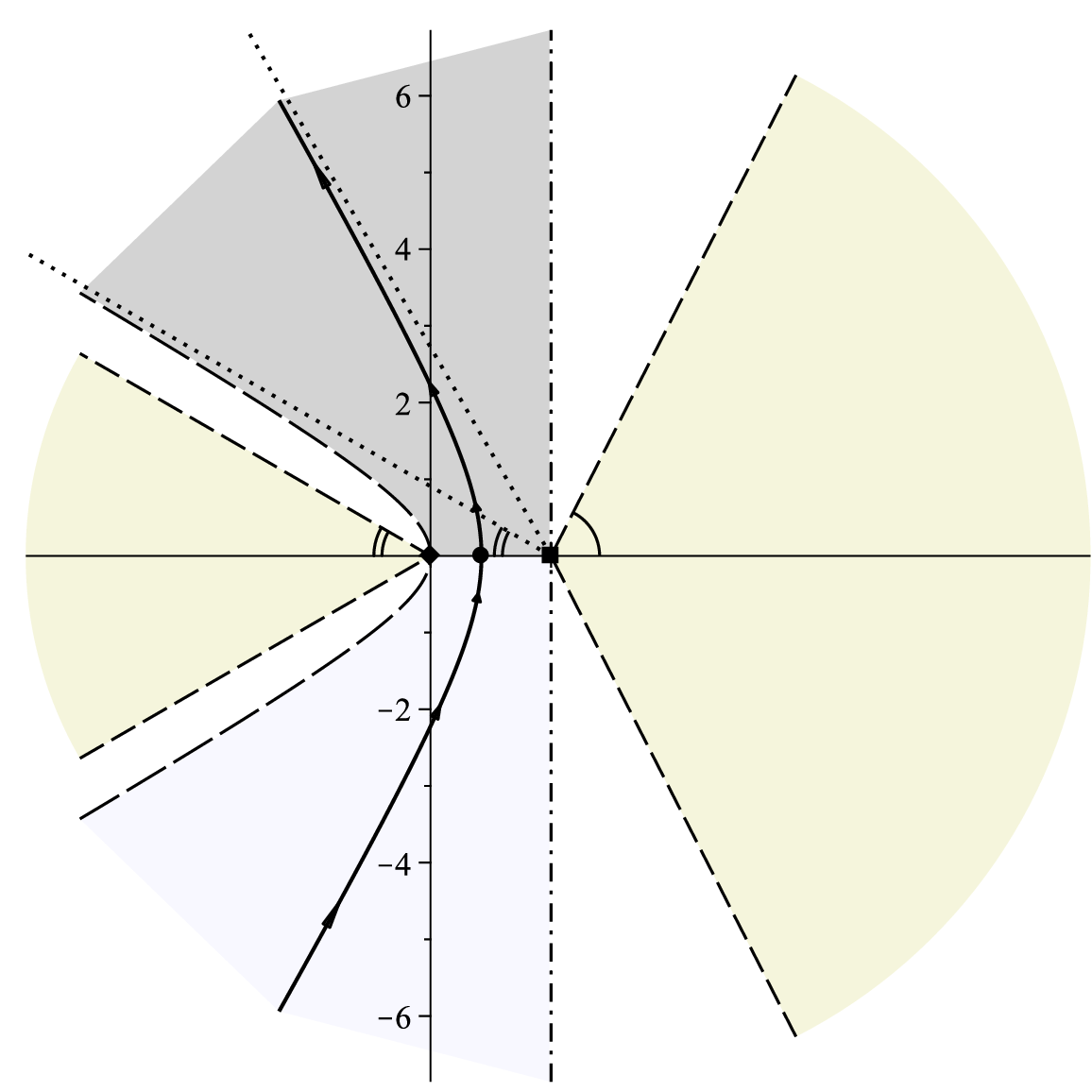}
				\put(20,140){\small $\mathrm{Sp}(-A)$}
				\put(270,195){\small $\varphi_s$}
				\put(590,210){\small $\phi_c$}
				\put(30,405){$\Gamma_s$}
				\put(200,620){$\Gamma_I$}
				\put(480,620){$\Gamma_c$}
				\put(0,660){\scriptsize ({\bf b})}
			\end{overpic}%
		\else
			\includegraphics[width=0.7\textwidth,permil=true,viewport=61 200 471 480,clip=true]
			{FCP_res_sec_cont_sphi_pi_6_sa_0_alpha_0_7_beta_1_ac_pi_2_phi_c_pi_2_n1}
		\fi
		\caption{Schematic plot of the domain $D \equiv D_d$ in which the parametrized integrand $\cF_{\alpha,1}(\beta,t,\xi)$ remains analytic and exponentially decaying for any $t \in [0, T]$ (\textbf{a}); and the image of $D_d$ under the mapping $v \to z(v)$ defined by $\Gamma_I$, along with the "forbidden" regions of complex plane indicated by ``beige'' color (\textbf{b}).
		The parameters of $\Gamma_I$: $\alpha =0.7$, $\beta=1$, $\varphi_s = {\pi}/{6}$, $\phi_c=0.35\pi$.
		}
		\label{fig:FCPML_hyp_cont}
	\end{centering}
\end{figure}
Additionally, we require that two limiting hyperbolas  $z(\xi \pm i d)$, with $d = \omega(\Theta)/2$, are asymptotically parallel to the rays forming the left and right boundary of $\Theta$.
Observe that, for any fixed $\eta \in [-d,d]$, the complex curve $z(\xi + i \eta)$, $\xi \in	(-\infty, \infty )$ is also a hyperbolic contour, feasible in the sense of  \cref{lem:FCP_SO_ML_repr}.
The described conditions on $\Gamma_I$ allows us to calculate the coefficients $a_I$, $b_I$ in the closed form. 
For the given $\alpha$, $\beta$, $\varphi_s$ and $\omega(\Theta)$, we define
\begin{equation}\label{eq:FCP_hyp_cont_par_final}
	\begin{aligned}
		a_I
		= \frac{a_0\cos{(\omega(\Theta)/2 - \phi_s)}}{\cos{\phi_s}} ,
		\quad b_I
		= \frac{a_0\sin{(\omega(\Theta)/2 - \phi_s)}}{\cos{\phi_s}},
	\end{aligned}
\end{equation}
where $a_0$ is chosen as $a_0 = {\pi}/{6}$ to ensure a reasonable separation between $\Gamma_I$ and the origin.
In Section 3.2 of \cite{Sytnyk2023} we proved that, the sinc-quadrature approximation of propagator representation \eqref{eq:FCP_SO_cont_repr}, parametrized by \eqref{eq:int_cont_hyp_crit_sec}, \eqref{eq:FCP_hyp_cont_par_final},  converges at the rate $\mathcal{O}(e^{-c\sqrt{\kappa N}})$, with $c = \sqrt{\pi \omega(\Theta)}$.
Our aim below, is to extend this result to the subordination based representation established by \cref{thm:FCP_SO_MLcor_cont_repr}.
By substituting $z = z(\xi)$, we transform formula \eqref{eq:FCP_SO_MLcor_cont_repr}  into
\begin{equation}\label{eq:FCP_SO_MLcor_repr_par}
	\begin{aligned}
		S_{\alpha}(t) x
		 & = \frac{1}{2\pi i} \intl_{-\infty}^{\infty}E_{\gamma,1}{\left(z(\xi) t^{\gamma}\right)} F_{\beta,1}(\xi)x \, d\xi + x, \\
		F_{\beta,1} (\xi)
		 & = z'(\xi)\left( z^{\beta-1}(\xi)\left (z^\beta(\xi) I + A \right )^{-1} - \frac{1}{z(\xi)}I \right),
	\end{aligned}
\end{equation}
where $z'(\xi) = -a_I\sinh(\xi) + ib_I\cosh(\xi)$.
The function $F_{\beta,1} (\xi)x$ is analytic in $D_d$ by construction of $\Gamma_I$ and, due to \cref{prop:FCP_res_cor}, its norm is exponentially decaying, when $\xi \to \infty$ in this region.
Next, for some fixed $h > 0$, $N \in \N$, we define a discretized version of \eqref{eq:FCP_SO_MLcor_repr_par}:
\begin{equation}\label{eq:FCP_SO_MLcor_sinc_quad}
	\begin{aligned}
		\wt{S}_{\alpha}^{N}(\beta, t) x
		 & = \frac{h}{2 \pi i}\sum_{k=-N}^{N} \cF_{\alpha,1} (\beta,t, kh) + x,
		\\
		\cF_{\alpha,1} (\beta, t, \xi)
		 & = E_{\gamma,1}{\left(z(\xi) t^{\gamma}\right)} F_{\beta,1} (\xi) x,
		\quad \gamma = \frac{\alpha}{\beta}.
	\end{aligned}
\end{equation}
Here, we made the dependence of $\wt{S}_{\alpha}^{N}(\beta, t)$ on $\beta$ explicit, because the value of $\beta$ has an impact on the approximation accuracy via the maximal angular size $\omega_m \geq \omega(\Theta)$.
The error of \eqref{eq:FCP_SO_MLcor_sinc_quad} admits the following two-term decomposition
\[
	\|S_{\alpha}(t)x - \wt{S}_{\alpha}^N(\beta, t)x \|
	\leq \|S_{\alpha}(t)x - \wt{S}_{\alpha}^\infty(\beta, t)x\| + \|\wt{S}_{\alpha}^\infty(\beta, t)x - \wt{S}_{\alpha}^N(\beta, t)x\|,
\]
where $\|\cdot\|$ is the norm of $X$, as before.
According to the general theory of sinc-quadrature \cite[Sec. 3.2]{Stenger1993}, the discretization error term $\|S_{\alpha}(t)x - \wt{S}_{\alpha}^\infty(\beta, t)x\|$ can be bounded as
\begin{equation}\label{eq:FCP_MLprop_discr_error}
	\left \|S_{\alpha}(t) x - \wt{S}_{\alpha,1}^\infty(\beta, t) x \right \|	\leq  \frac{e^{-\pi d/h}}{2 \sinh (\pi d/h)}\|{\cF_{\alpha,1}}(\beta, t, \cdot)\|_{{\bf H}^1(D_d)}, \quad t \in [0,T].
\end{equation}
${\bf H}^1(D_d)$ denotes the Hardy space of analytic in $D_d$ functions $\mathcal{F}: {\mathbb C} \rightarrow X$ equipped with the norm
\[
	\| {\mathcal{F}} \|_{{\bf H}^1 (D_d)}
	= 	\lim\limits_{\epsilon \rightarrow 0}\intl_{\partial D_d\left (\epsilon \right )}\|{\mathcal F}(z)\||dz| ,
\]
where
$
	D_d(\epsilon)=\{z \in \mathbb{C}:\; | \Re(z)| < 1/\epsilon, \
	|\Im(z)|<d(1-\epsilon)\}
$
and $\partial D_d(\epsilon)$ is the boundary of $ D_d(\epsilon)$.
The estimate for $\|{\cF_{\alpha,1}}(\beta, t, \cdot)\|_{{\bf H}^1(D_d)}$ is provided by the next lemma.

\begin{lemma} \label{lem:FCP_SO_MLISO_Hp_bound}
	Assume that operator $A$ and parameters $\alpha, \beta$ satisfy the conditions of \cref{lem:FCP_SO_ML_repr}.
	Then, for any $x \in D(A^\kappa)$, $\kappa > 0$, and arbitrarily small $\delta>0$
	\begin{equation}\label{eq:FCP_SO_MLcor_repr_norm_est}
		\begin{aligned}
			\|{\cF_{\alpha,1}}(\beta, t, \cdot)\|_{{\bf H}^1(D_{d-\delta})}
			 & \leq  \frac{C_{\beta,1}^\pm}{\beta \kappa } \left(1 + e^{a_m t}\right) \|A^{\kappa}x\|,
		\end{aligned}
	\end{equation}
	where  $a_m= \left(a_0  - a_0 \frac{\cos{\phi_c}}{\cos{\phi_s}}\right)^{\beta/\alpha}$ and $C_{\beta,1}^\pm > 0$ is a bounded constant that depends on $\Gamma_I, d, \delta, \beta$ and does not depend on $t \in [0, T]$.
\end{lemma}
\begin{proof}
	To prove bound \eqref{eq:FCP_SO_MLcor_repr_norm_est} we utilize Lemma 2 from \cite{Sytnyk2023}, associated with representation \eqref{eq:FCP_SO_cont_repr} on the same contour.
	Recall that the image of $z(D_d)$, $d \in (0, \omega_m/2]$  is contained within $\Theta_m$, whose right-upper boundary is inclined at the angle greater or equal than $\phi_c \equiv \pi \gamma /2$ (see \cref{fig:FCPML_hyp_cont} (b)).
	Hence, for arbitrary $w \equiv \xi + i \nu \in D_d(\delta)$, we get $\Re{z(w)^{1/\gamma}} \leq \Re{z(-i {\omega_m}/{2})^{1/\gamma}}$, with
	\[
		\begin{aligned}
			z\left(-i \frac{\omega_m}{2}\right)
			 & = a_0 - \frac{a_0}{\cos{\phi_s}}\left(
			\cos{\left(\frac{\omega_m}{2} - \phi_s \right )} \cos{\frac{\omega_m}{2}} - \sin{\left(\frac{\omega_m}{2} - \phi_s\right)} \sin{\frac{\omega_m}{2}}
			\right)                                                                              \\
			 & = a_0 \left(1 - \frac{\cos{\left(\omega_m - \phi_s\right)} }{\cos{\phi_s}}\right)
			= a_0 \left(1 - \frac{\cos{\phi_c}}{\cos{\phi_s}}\right).
		\end{aligned}\]
	When combined with estimate \eqref{eq:ML1_bound}, the above formula yields
	\begin{equation}\label{eq:ML1_bound_Dd}
		\left| E_{\gamma,1}{\left(z(w) t^{\gamma}\right)} \right|
		\leq M_E\left (1+e^{t\Re{z(w)^{1/\gamma}}} \right )
		\leq M_E (1+e^{a_m t}),
	\end{equation}
	where $M_E = \max\{M_1, M_2\}$ and $a_m$ is defined in the lemma's premise.
	The rest of the proof is straightforward:
	we apply \eqref{eq:ML1_bound_Dd} to the first term of the inequality
	\[
		\left\| \cF_{\alpha,1} (\beta, t, w) \right\|
		\leq \left| E_{\gamma,1}{\left(z(w) t^{\gamma}\right)} \right| \left |\frac{z'(w)}{z(w)} \right | \frac{(1+M) K}{(1+|z(w)|^\beta)^\kappa} \left\|A^\kappa x\right\|,
	\]
	obtained from \eqref{eq:FCP_SO_MLcor_sinc_quad} via \eqref{eq:FCP_res_cor_norm_est}, and then use the bounds from \cite[Lem. 2]{Sytnyk2023} to estimate the remaining terms.
	This procedure yields an inequality
	\begin{equation}\label{eq:FCP_SO_MLcor_repr_norm_est_der}
		\left\| \cF_{\alpha,1} (\beta, t, w) \right\|
		\leq
		C_{\beta,1}^\pm \left(1 + e^{a_m t}\right)
		e^{- \kappa\beta|\xi|}\|A^{\kappa}x\|,
	\end{equation}
	transformed into desired bound \eqref{eq:FCP_SO_MLcor_repr_norm_est} by the integration along $\partial D_d(\delta)$.
	The constant $C_{\beta,1}^\pm= M_E\left(C_{\beta,1}(\kappa, \delta -d) + C_{\beta,1}(\kappa, d - \delta )\right)$ in formulas \eqref{eq:FCP_SO_MLcor_repr_norm_est} and \eqref{eq:FCP_SO_MLcor_repr_norm_est_der} is defined as
	\begin{equation}\label{eq:FCP_SO_exp_cor_repr_norm_const}
		C_{\beta,1}(\kappa,\nu) =
		\frac{K_1 b( \nu)}{ \left(a( \nu) - a_0\right) r_0^\kappa(\nu)}, \quad r_0(\nu)  = \inf\limits_{\xi \in \R} r(\xi, \nu).
	\end{equation}
	Where $K_1 > 0$ and $r(\xi, \nu)$ is the solution of equation $1+|z(w)|^\beta = r(\xi,\nu)\cosh^\beta{\xi}$.
	Meanwhile, $a(\nu)$ and $b(\nu)$ denote the pair of coefficients for the parametric family of hyperbolas $\Gamma(\nu)  = \{a_0 - a(\nu) \cosh{\xi} + ib(\nu) \sinh{\xi}: \; \xi\in(-\infty,\infty)\}$, obtained from \eqref{eq:int_cont_hyp_crit_sec} by the argument substitution $\xi = \xi + i \nu$:
	\begin{equation}\label{eq:FCPML_ab_mu}
		a(\nu) = a_I \cos{\nu}+b_I\sin{\nu},	 \quad
		b(\nu) = b_I \cos{\nu}-a_I\sin{\nu}.
	\end{equation}
\end{proof}
Next, we estimate the truncation error $\|\wt{S}_{\alpha}^\infty(\beta, t)x - \wt{S}_{\alpha}^N(\beta, t)x\|$ of \eqref{eq:FCP_SO_MLcor_sinc_quad}.
\begin{lemma}\label{lem:FCP_SO_MLISO_trunc_bound}
	Assume that operator  $A$ and parameters $\alpha, \beta$ satisfy the conditions of \cref{lem:FCP_SO_ML_repr}.
	Then, for any $x \in D(A^\kappa)$, $\kappa>0$, the truncation error of $(2N+1)$-term approximations \eqref{eq:FCP_SO_MLcor_sinc_quad} with the step-size $h > 0$ satisfies the estimate
	\begin{equation}\label{eq:FCP_SO_MLcor_repr_cont_est}
		\|\wt{S}_{\alpha}^\infty(\beta, t)x - \wt{S}_{\alpha}^N(\beta, t)x\|
		\leq  \frac{M_E C_{\beta,1} (\kappa,0)(1 + e^{a_m t})}{\pi \kappa \beta}e^{-\kappa\beta Nh}\|A^{\kappa}x\|,
	\end{equation}
	where $M_E$, $a_m$ and $C_{\beta, 1}(\kappa, \nu)$ are defined by  \cref{lem:FCP_SO_MLISO_Hp_bound}.
\end{lemma}
\begin{proof}
	We proceed by applying estimate \eqref{eq:FCP_SO_MLcor_repr_norm_est_der} to the left-hand side of \eqref{eq:FCP_SO_MLcor_repr_cont_est}
	\[
		\begin{aligned}
			\frac{h}{2\pi}\left \|\sum_{|k|>N}\!\!\cF_{\beta,1}\left (\beta, t, kh\right ) \right \|
			\leq \frac{h}{\pi}M_E C_{\beta,1} (\kappa,0) \left(1 + e^{a_m t} \right)
			\sum_{k=N+1}^{\infty} 	e^{- \kappa \beta k h}\|A^{\kappa}x\|
			\\
			\leq \frac{h M_E C_{\beta,1} (\kappa,0)(1 + e^{a_m t})}{\pi (1 - e^{-\kappa \beta h})e^{\kappa \beta (N+1) h}}\|A^{\kappa}x\|
			\leq \frac{M_E C_{\beta,1} (\kappa,0)(1 + e^{a_m t})}{\pi \kappa \beta e^{\kappa \beta N h}}\|A^{\kappa}x\|.
		\end{aligned}
	\]
\end{proof}
Now, we are fully equipped to formulate the main result regarding the accuracy of the subordination based approximation $\wt{S}_{\alpha}^N(\beta, t)$, defined by \eqref{eq:FCP_SO_MLcor_sinc_quad}.
\begin{theorem}\label{thm:FCPML_prop_appr}
	Let $A$ be a strongly positive operator with the domain $D(A)$ and the spectrum $\mathrm{Sp}(A) \subset \Sigma(\rho_s, \varphi_s)$,  $\rho_s>0$, $\varphi_s < \pi/{2}$.
	Then, for any $t \in [0, T]$, $x \in D(A^\kappa)$, $\kappa > 0$, and  $\alpha, \beta \in (0,2)$, such that $0< \alpha \leq \min{\left\{\beta, 2 \left (1 - {\varphi_s}/{\pi} \right) \right\} }$, the error of sinc quadrature--based approximation
	$\wt{S}_{\alpha,1}^N(\beta, t)x$, $N \in \N$  satisfies the bound
	\begin{equation}\label{eq:FCP_SO_MLcor_err_est}
		\left \|S_\alpha(t)x - \wt{S}_{\alpha}^N(\beta, t)x \right \| \leq
		{C_1} \frac{\left(1 + e^{a_m t}\right)}{ \kappa \beta }	\exp{\left(-c\sqrt{\kappa\beta N}\right)}
		\|A^{\kappa}x\|,
	\end{equation}
	with $a_m= \left(a_0  - a_0 \frac{\cos{\phi_c}}{\cos{\phi_s}}\right)^{\beta/\alpha}$,  $\phi_c = \frac{\pi\alpha}{2\beta}$, $\phi_s = \min\{\pi, \frac{\pi -\varphi_s}{\beta}\}$ and $c=\sqrt{\pi \omega(\Theta)}$, provided that the step-size $h$ in  \eqref{eq:FCP_SO_MLcor_sinc_quad} is chosen as	$h=\sqrt{\frac{\pi \omega(\Theta)}{\kappa \beta N}}$,
	Here, $a_0 > 0$ and $\omega(\Theta) \in (0, \omega_m]$ are the given contour shift  and angular size parameters from \eqref{eq:int_cont_hyp_crit_sec} and \eqref{eq:FCP_hyp_cont_par_final}, correspondingly.
	The constant $C_1$ in \eqref{eq:FCP_SO_MLcor_err_est} is independent of $t,N$.
\end{theorem}
\begin{proof}
	In order to estimate $\left \|S_\alpha(t)x - \wt{S}_{\alpha,1}^N(\beta, t)x \right \|$ we combine discretization error bounds \eqref{eq:FCP_MLprop_discr_error}, \eqref{eq:FCP_SO_MLcor_repr_norm_est} with the truncation error bound given by \eqref{eq:FCP_SO_MLcor_repr_cont_est}:
	\[
		\begin{aligned}
			\|S_{\alpha}(t)x - \wt{S}_{\alpha}^N(\beta, t)x  \|
			 & \leq \frac{\left(1 + e^{a_m t}\right)}{ \kappa \beta }  \left( \frac{C_{\beta,1}^\pm e^{-\frac{\pi d}{h}}}{2 \sinh{\frac{\pi d}{h}}} +  \frac{M_E C_{\beta,1} (\kappa,0)}{\pi }e^{-\kappa\beta Nh}\right)\|A^{\kappa}x\| \\
			 & \leq  \frac{\left(1 + e^{a_m t}\right)}{ \kappa \beta } \left(\frac{c_0 C_{\beta,1}^\pm }{e^{2\frac{\pi d}{h}}} + \frac{M_EC_{\beta,1} (\kappa,0)}{\pi e^{\kappa\beta Nh}}  \right)
			\|A^{\kappa}x\|,
		\end{aligned}
	\]
	and then equate the arguments of two exponents inside the brackets.
	This yields the expression $h =\sqrt{\frac{\pi \omega(\Theta)}{\kappa \beta N}}$.
	After back-substitution of $h$ into the last estimate we finally arrive at \eqref{eq:FCP_SO_MLcor_err_est}, with $C_1 = c_0 C_{\beta,1}^\pm  + M_E C_{\beta,1} (\kappa,0)$ and  $c_0 = \left(1-\e^{-c\sqrt{\kappa\beta}} \right)^{-1}$.
\end{proof}

The reader should note that the principal part of error estimate \eqref{eq:FCP_SO_MLcor_err_est}  has $\beta$ in its argument, whereas the convergence rate of the former method, stemming from \eqref{eq:FCP_SO_cont_repr}, is determined by $\alpha$ (see \cite[Thm. 2]{Sytnyk2023}).
This implies that, the approximation quality $\wt{S}_{\alpha}^N(\beta, t)$ will not degrade like in \cite{Sytnyk2023}, when $\alpha$ decreases.
Furthermore, the convergence speed should increase in the case of the maximal angular size: $\omega(\Theta) = \phi_s - \phi_c$, since $\phi_c \to 0$, as $\alpha \to 0$.
If one, instead, sets $\omega(\Theta)$ to the sub-optimal value $\omega_\star = \phi_s -  \tfrac{\pi}{2} \max{\left\{1, \tfrac{1}{\beta}\right\}}$,
then the convergence order in \eqref{eq:FCP_SO_MLcor_err_est} remains fixed for all $\alpha \in (0, \beta]$.
The same is true regarding the contour $\Gamma_I$, which is invariant with respect to $\alpha$, if $\omega(\Theta) = \omega_\star$.
This fact theoretically justifies the possibility of the resolvent reuse in the scenario with evaluation of the  solution for multiple $\alpha,t$, mentioned at the end of \cref{sec:FCPML_subordinated_approximation}.
Another essential feature of $\wt{S}_{\alpha}^N(\beta, t)$, $\alpha \in (0,2)$, is its dependence only on the scalar values of the Mittag-Leffler function $E_{\gamma,1}(s)$, with $s \in \Gamma_I$, $\gamma \leq 1$.
Efficient numerical methods for the evaluation of such functions from \cite{Garrappa2015,McLean2021} are reliant upon the contour representation of  $E_{\gamma,1}(s)$ and the robust singularity identification, which are currently tractable only for $\gamma \leq 1$.
Notwithstanding the above distinctions, we emphasize that approximation \eqref{eq:FCP_SO_MLcor_sinc_quad} is the proper extension of the one from \cite{Sytnyk2023}.
Hence, these two approximations can be used interchangeably in applications.
We will take advantage of this observation in the next section, where the current method is applied to  the mild solution of \eqref{eq:FCP_DEBC}.

\section{Applications and Numerical examples}\label{sec:FCPML_Applications_Examples}
In this section, we show how to incorporate the new propagator approximation formula into the numerical solution scheme for problem \eqref{eq:FCP_DEBC} developed in \cite{Sytnyk2023}.
This process requires only nonessential modifications, mostly limited to the second-initial-condition part of \eqref{eq:FCP_InhomSol_rep}.

We begin by stating a useful identity for $E_{\gamma,1}(z)$:
\[
	\intl_{0}^t E_{\gamma,1}{\left(z \left (t-s \right )^{\gamma}\right)}\, ds
	= t E_{\gamma,2}{\left(z t^{\gamma}\right)}
	= \frac{t^{1-\gamma}}{z} \left (E_{\gamma,2-\gamma}{\left(z t^{\gamma}\right)}-\frac{1}{\Gamma(2-\gamma)} \right ),
\]
derived with help of (4.2.3) and (4.4.4) from \cite{Gorenflo2020}.
The strong continuity of $S_\alpha(t)$ permits us to apply Fubini's theorem \cite[Thm. 8.7]{Lang2012} to the contour integral representation of $\intl_{0}^t \!\! S_\alpha(t-s) u_1\, ds$  obtained by means of  \eqref{eq:FCP_SO_MLcont_repr}, and, then, make use of the above identity for $E_{\gamma,1}(z)$.
As a result, we get
\[
	\intl_{0}^t \!\! S_\alpha(t-s) u_1\, ds
	= \frac{t^{1-\gamma} }{2 \pi i } \left( \ \intl_{\Gamma_I}\!\! E_{\gamma,2-\gamma}{\left(z t^{\gamma}\right)}  \frac{(z^\beta I + A)^{-1}u_1} {z^{2-\beta}} \, dz -  \intl_{\Gamma_I}  \frac{(z^\beta I + A)^{-1}u_1} {\Gamma(2-\gamma)z^{2-\beta}} \, dz \!\right).
\]
The second integral inside the brackets is equal to zero by Corollary 1 from \cite{Sytnyk2023b}.
Consequently, the homogeneous part $u_\mathrm{h}$ of \eqref{eq:FCP_InhomSol_rep}
can be rewritten as
\[
	u_\mathrm{h} = S_\alpha(t) u_0 + S_{\alpha,2}(t) u_1.
\]
Where $S_\alpha(t) u_0$ is expressed by \eqref{eq:FCP_SO_MLcor_cont_repr}  and $S_{\alpha,2}(t) u_1$ is defined by
\begin{equation}\label{eq:FCP_ISO_ML_cont_repr}
	S_{\alpha,2}(t) u_1
	= \frac{t^{1-\gamma} }{2 \pi i } \intl_{\Gamma_I}  E_{\gamma,2-\gamma}{\left(z t^{\gamma}\right)} z^{\beta-2} (z^\beta I + A)^{-1}u_1 \, dz,
\end{equation}
with $0 < \alpha \leq \beta < 2$, as before.
Bounds \eqref{eq:ResSector} and \eqref{eq:ML1_bound} stipulate a quadratic decay of the integrand from \eqref{eq:FCP_ISO_ML_cont_repr} in the region $\Theta_m$.
Hence, the operator function $S_{\alpha,2}(t) u_1$ is well-defined for any $u_1 \in X$ and, after parametrization $z= z(\xi)$, its integral representation can be discretized analogously to \eqref{eq:FCP_SO_MLcor_repr_par}.

For any fixed $h \in \R$, $N \in \N$, we define
\begin{equation}\label{eq:FCP_ISO_exp_cor_repr_par}
	\begin{aligned}
		\wt{S}_{\alpha,2}^N(\beta, t) u_1
		 & = \frac{h }{2 \pi i}\sum_{k=-N}^{N} \cF_{\alpha,2} (\beta, t, kh),
		\quad \gamma = \frac{\alpha}{\beta},                                                                                         \\
		\cF_{\alpha,2} (\beta, t, \xi)
		 & = t^{1-\gamma}  E_{\gamma,2-\gamma}{\left(z(\xi) t^{\gamma}\right)}  z'(\xi)z^{\beta-2}(\xi)(z^\beta(\xi)I + A)^{-1} u_1,
	\end{aligned}
\end{equation}
where $z(\xi)$ is determined by \eqref{eq:int_cont_hyp_crit_sec}, \eqref{eq:FCP_hyp_cont_par_final} and the user-chosen $a_0 > 0$, $\omega(\Theta) \in (0, \omega_m]$.

\begin{theorem}\label{thm:FCPML_hom_sol_appr}
	Assume that operator  $A$ and parameters $\alpha$, $\beta$, $a_0$, $a_m$,  $c$, $\omega(\Theta)$ are defined as in \cref{thm:FCPML_prop_appr}.
	Then, for any $u_0 \in D(A^\kappa)$, $\kappa>0$ and $u_1 \in X$, the approximate solution $\wt{u}_{\mathrm{h}}^N(t) = \wt{S}_{\alpha}^{N_1}(\beta, t) u_0 + \wt{S}_{\alpha,2}^{N_2}(\beta, t) u_1$, with $N_1 = N$, $N_2 = \lceil \kappa \beta N\rceil$, converges to the homogeneous part of mild solution  \eqref{eq:FCP_InhomSol_rep} and the following error bound is valid
	\begin{equation}\label{eq:FCP_hom_sol_err_est}
		\left\| u_{\mathrm{h}}(t) - \wt{u}_\mathrm{h}^N(t) \right\|  \leq {C_\kappa}{(1+ t^{1-\gamma}) (1+e^{a_m t}) }
		\exp{\left(-c\sqrt{\kappa\beta N}\right)}
		\|A^{\kappa}u_0\|,
	\end{equation}
	provided that the step-sizes in \eqref{eq:FCP_SO_MLcor_sinc_quad} and \eqref{eq:FCP_ISO_exp_cor_repr_par} are set as $h_1 = h_2 =\sqrt{\frac{\pi \omega(\Theta)}{\kappa \beta N}}$.
	The constant $C_{\kappa}$ is dependent on $A$, $u_0$, $u_1$, and independent of $t,N$.
\end{theorem}
\begin{proof}
	The error of the propagator approximation $\wt{S}_{\alpha}^{N_1}(\beta, t) u_0$ was characterized by \cref{thm:FCPML_prop_appr}.
	Here, we focus on deriving an estimate for $\left\|S_{\alpha,2}(t) u_1 - \wt{S}_{\alpha,2}^{N}(\beta, t) u_1\right\|$, using the technique adopted in the proofs of \cref{lem:FCP_SO_MLISO_Hp_bound,lem:FCP_SO_MLISO_trunc_bound}.
	Bearing this goal in mind, we take a closer look at the estimate
	\[
		\left\| \cF_{\alpha,2} (\beta, t, w) \right\|
		\leq M_E t^{1-\gamma}  (1+e^{a_m t}) \left |\frac{z'(w)}{z^2(w)} \right |  \frac{|z(w)|^{\beta} }{1+|z(w)|^\beta}
		\left\| u_1\right\|, \quad w \equiv \xi + i \nu \in D_d(\delta),
	\]
	obtained by applying \eqref{eq:FCP_res_cor_norm_est}, \eqref{eq:ML1_bound_Dd} to the norm of $\cF_{\alpha,2} (\beta, t, w)$.
	For any $d \in (0, \omega_m/2]$ and $\delta \in [0,d]$, this estimate differs from the corresponding estimate for $\|\cF_{\beta,2} (t, w)\|$, obtained in the proof of Lemma 2 from \cite{Sytnyk2023}, only by the factor $M_E  t^{1-\gamma}   (1+e^{a_m t})e^{-a_0 t}$.
	This fact allows us to apply the cited lemma directly, leading to the bounds
	\begin{eqnarray}
		&\left\| \cF_{\alpha,2} (\beta, t, w) \right\|
		\leq C_{\beta,2}(\nu) t^{1-\gamma}  (1+e^{a_m t})
		\|u_1\|,
		\label{eq:FCP_ISO_ML_cont_repr_norm_est}
		\\
		&\left\|{\cF_{\alpha,2}}(\beta, t, \cdot)\right\|_{{\bf H}^1(D_{d-\delta})}
		\leq  C_{\beta,2}^\pm (\delta) t^{1-\gamma}  (1+e^{a_m t}) \|u_1\|,\hspace*{3.2em}
		\label{eq:FCP_ISO_ML_cont_repr_Hp_norm_est}
	\end{eqnarray}
	where $C_{\beta,2}^\pm  (\delta) = M_E \left(C_{\beta,2}(\delta -d) + C_{\beta,2}(d - \delta )\right)$, with some constant $M_E > 0$, and
	\[
		C_{\beta,2}(\nu)  =
		K_2\frac{b(\nu) \left(b^2(\nu)+(a(\nu) - a_0)^2 \right)^{\beta/2}}{(a(\nu) - a_0)^2  r_0(\nu)},
	\]
	with $r_0(\nu)$  and $a(\nu)$, $b(\nu)$ being defined by \eqref{eq:FCP_SO_exp_cor_repr_norm_const}  and \eqref{eq:FCPML_ab_mu}, correspondingly.
	Furthermore, using \eqref{eq:FCP_ISO_ML_cont_repr_norm_est} and the derivation procedure from the proof of \cref{lem:FCP_SO_MLISO_trunc_bound}, for the truncation error of approximation \eqref{eq:FCP_ISO_exp_cor_repr_par} we get
	\[
		\frac{h}{2\pi}\left \|\sum_{|k|>N}\!\!\cF_{\beta,2}\left (\beta, t, kh\right ) \right \|
		\leq \frac{M_E C_{\beta,2} (0)}{\pi }\frac{t^{1-\gamma}  (1+e^{a_m t}) } {e^{(N + 1) h}}\|u_1\|.
	\]
	The estimate $\left \|S_{\alpha}(t) x - \wt{S}_{\alpha,2}^\infty(\beta, t) x \right \|	\leq  \frac{e^{-\pi d/h}}{2 \sinh (\pi d/h)}\|{\cF_{\alpha,2}}(\beta, t, \cdot)\|_{{\bf H}^1(D_d)}$
	together with bound \eqref{eq:FCP_ISO_ML_cont_repr_Hp_norm_est} and the foregoing estimate for the truncation error yield
	\begin{equation}\label{eq:FCP_ISO_ML_err_est}
		\left\|S_{\alpha,2}(t) u_1 - \wt{S}_{\alpha,2}^{N}(\beta, t) u_1\right\|
		\leq {C_2} t^{1-\gamma}  (1+e^{a_m t}) \exp{\left(-c\sqrt{N}\right)} \|u_1\|,
	\end{equation}
	with $C_2 = c_0 C_{\beta,2}^\pm  + M_E C_{\beta,2} (\kappa,0)$ and $h=\sqrt{{\pi \omega(\Theta)}/{N}}$, obtained by the similar means as in \cref{thm:FCPML_prop_appr}.
	In the last step, we equate the order of decay in \eqref{eq:FCP_SO_MLcor_err_est} and \eqref{eq:FCP_ISO_ML_err_est} by setting the discretization parameters of $\wt{u}_{\mathrm{h}}^N(t)$ to $N_1 = N$, $N_2 = \lceil \kappa \beta N\rceil$.
	This will produce estimate \eqref{eq:FCP_hom_sol_err_est} with the constant $C_{\kappa} = \max{\left\{C_1, C_2 \|u_1\|/\|A^{\kappa}u_0\|\right\} }$.
\end{proof}

Let us again underscore that the convergence order of the newly derived numerical scheme for $u_\mathrm{h}(t)$  is asymptotically equal to $\cO\left (e^{-c\sqrt{\kappa\beta N}}\right )$
irrespective of the specific choice of $\alpha \in (0, \beta]$.
In this regard, \cref{thm:FCPML_hom_sol_appr}  can be viewed as a generalization of the earlier results \cite{gm5,bGavrilyuk2011} devoted to $\alpha = 1$, and of the corresponding result from \cite{Sytnyk2023}, recovered by setting $\beta = \alpha$.
The mentioned numerical methods are based on the time-independent contour $\Gamma_I$, optimized towards the stable and efficient evaluation of $u_\mathrm{h}(t)$ for multiple values of $t \in [0, T]$.
These are not to be confused with another class of similar methods \cite{lopez-fernandez1,Weideman2010,McLean2010}, that pursue faster theoretical convergence $\cO\left (e^{-C N/\ln{N}}\right )$ for a fixed $t \in (t_0, T]$, by choosing the contour $\Gamma_I$ in \eqref{eq:FCP_SO_cont_repr} to be both $t$ and $N$ dependent.
Unfortunately, such choice of $\Gamma_I$  for $u_\mathrm{h}(t)$ is thoroughly studied only under strict conditions: $t_0 > 0$, $u_0 \in D(A)$, making these methods unsuitable for our application scenario.

Next, we approximate the inhomogeneous part $u_{\mathrm{ih}}(t)$ of \eqref{eq:FCP_InhomSol_rep}, by combining its discretized representation $\wt{u}_{\mathrm{ih}}^N(t)$, derived in \cite{Sytnyk2023}, with subordination based approximation \eqref{eq:FCP_SO_MLcor_sinc_quad}.
For a fixed $N \in \N$, let
\begin{equation}\label{eq:FCP_inhom_sol_appr}
	\begin{aligned}
		\wt{u}_{\mathrm{ih}}^N(t)  = & \wt{J}_\alpha^{N_0} \wt{S}_{\alpha}^N(\beta, t)f(0) + h_1 \sum\limits_{k=-N_1}^{N_1}  \mathcal{G}_{\alpha,\beta}^{N_2}(0, t, kh_1) \\
		                             & + \frac{h_3 h_4}{2\pi i}\sum\limits_{\ell =-N_3}^{N_3}F_{\beta,1} (\ell h_3)
		\sum\limits_{k=-N_4}^{N_4} \mathcal{G}_{\alpha,\beta}^{N_5}(z(l h_3), t, kh_4),
	\end{aligned}
\end{equation}
where $N_i = N_i(f, \alpha, \beta, N)$, $h_i>0$, $i=0,\ldots 5$,  the function $F_{\beta,1} (\xi)$ is defined by \cref{eq:FCP_SO_MLcor_repr_par},
\[
	\begin{aligned}
		\mathcal{G}_{\alpha,\beta}^N(z, t, p) & =
		t\psi'(p) E_{\gamma,1}{\left(z t^\gamma (1-\psi(p))^\gamma\right)} \wt{J}_\alpha^{N}  f'\left ( t\psi(p) \right ), \quad \psi(p) = \frac{e^p}{1+e^p},
	\end{aligned}\]
and $\wt{J}_\alpha^{N}$ is the sinc-quadrature approximation of the Riemann--Liouville operator proposed in \cite[Prop. 2]{Sytnyk2023}.
The convergence of $\wt{u}_{\mathrm{ih}}^N(t)$ is characterized by the following result.
\begin{corollary}\label{thm:FCP_ML_inhom_sol_appr}
	Assume that operator  $A$ and parameters $\alpha$, $\beta$, $a_0$, $a_m$,  $c$, $\omega(\Theta)$ are defined as in \cref{thm:FCPML_prop_appr}.
	If the right-hand side $f(t)$ of \eqref{eq:FCP_DE} admits the analytic extension into the "eye-shaped" domain \cite{Stenger1993}: $z \in D_d^2$, $d \in (0, \pi/2)$, and $f(0),f'(z) \in D(A^{\chi})$, for some $\chi > 0$, then the  error of approximation $\wt{u}_{\mathrm{ih}}^N(t)$ from \eqref{eq:FCP_inhom_sol_appr} satisfies the bound
	\begin{equation}\label{eq:FCP_inhom_sol_err_est}
		\left \|u_{\mathrm{ih}}(t) - \wt{u}_{\mathrm{ih}}^N(t) \right \|
		\leq
		{C_{f,\chi}}  L(t)
		\exp{\left( -c\sqrt{\beta\chi N}\right)},
	\end{equation}
	with $L(t) = \left(\frac{t}{\beta \chi} + \frac{1+t}{\Gamma(\alpha)}t^\alpha +  \frac{\chi + t(1 + \chi) }{\chi} t^\alpha  (1+e^{a_m t})   \right)$ and $h = h_i =\sqrt{\frac{\pi  \omega(\Theta)}{\beta \chi N}}$, $i = 0, \ldots 5$,  provided that the values of $N_i$ in \eqref{eq:FCP_inhom_sol_appr}  are chosen as
	\begin{equation}\label{eq:FCP_inhom_sol_N}
		N_1 = N_4 = \lceil{\beta\chi N}\rceil,
		\quad
		N_3 = N,
		\quad
		N_0 = N_2 = N_5 = \left\lceil \frac{\beta\chi N}{\min\left \{1, \alpha\right \}} \right\rceil.
	\end{equation}
	The constant $C_{\chi,f}$ is independent of $t,N$.
\end{corollary}
The proof of \cref{thm:FCP_ML_inhom_sol_appr} repeats the proof of Theorem 3  from \cite{Sytnyk2023}, with two amendments: $\chi \alpha \to \chi \beta$, ${e^{a_0 t}} \to {1+e^{a_m t}}$, caused by the switch to the subordination based approximation for $S_{\alpha}(t)$.
It will be omitted here for brevity.

As we can see from \eqref{eq:FCP_inhom_sol_err_est}, the error decay of $\wt{u}_{\mathrm{ih}}^N(t)$ is also controlled by $\beta$.
Hence, it is now possible for us to reach the target solution accuracy without the need to increase $N$ as in \cite{Sytnyk2023}, when progressively smaller fractional orders $\alpha$ are considered.
Computationally such process is not entirely free, because the parameters $N_i$, $i=0,2,5$,  from \eqref{eq:FCP_inhom_sol_N} grow proportionally to $1/\alpha$, driven by the kernel singularity of $J_\alpha$ in \eqref{eq:FCP_InhomSol_rep}.
This feature of \cref{thm:FCP_ML_inhom_sol_appr} inflicts a small but discernible increase of the round-off error of $\wt{u}_{\mathrm{ih}}^N(t)$, depicted in \cref{fig:FCPML_Ex1_inhom_err_vs_N}(a).

It is worthwhile to mention that, unlike earlier methods from \cite{McLean2010,Colbrook2022a},  formula \eqref{eq:FCP_inhom_sol_appr} requires only the knowledge of $f(0)$, $f'(t)$, $t \in (0, T]$.
Moreover, it permits for $f'(t)$ to have an integrable singularity at $t=0$, which can be handled without modifications to \eqref{eq:FCP_inhom_sol_appr} -- \eqref{eq:FCP_inhom_sol_N}, if $|f'(t)| < t^{\alpha-1}$.
For more detailed discussion about the computational capabilities of the numerical schemes for $\wt{u}_{\mathrm{h}}^N(t)$, $\wt{u}_{\mathrm{ih}}^N(t)$ and their algorithmic implementation, we refer the reader to Sections 3.3--3.4 of \cite{Sytnyk2023}.

The successful application of the developed schemes is contingent upon the ability to evaluate the Mittag-Leffler function $E_{\gamma,\sigma}(s)$, involved in the representations of $\wt{S}_{\alpha}^{N}(\beta, t)$ and $\wt{S}_{\alpha,2}^N(\beta, t)$, to the required precision.
We delegate this task to the optimal parabolic contour (OPC) algorithm from \cite{Garrappa2015}.
Functionally, OPC also relies on the residue-assisted quadrature of the integral in \eqref{eq:FCP_SO_cont_repr},  optimized for each given scalar input $A = s$.
It is, therefore, capable of handling large complex $s \in \Gamma_I$, if $\Re{s}$ is bounded.
\begin{remark}\label{rem:FCPML_OPC_omega}
	The OPC algorithm is able to numerically evaluate $E_{\gamma,\sigma}(s)$ in \eqref{eq:FCP_SO_MLcor_sinc_quad} and \eqref{eq:FCP_ISO_exp_cor_repr_par} for any $N \in \N$,  $\gamma \in (0, 1]$, $\sigma \in [1, 2)$,  if the parameter $\omega(\Theta)$ from \eqref{eq:FCP_hyp_cont_par_final} satisfies the inequality $\omega(\Theta) \leq \omega_{c}$, with $\omega_{c} = \phi_s - \max{\left\{{\pi\gamma}/{2}, \pi - (\pi - \varphi_s)/{\beta}\right\}}$.
\end{remark}
Clearly, the OPC condition $\omega(\Theta) \leq \omega_{c}$ is more restrictive than the constraint $\omega(\Theta) \leq \omega_{m}$, enforced by representation \eqref{eq:FCP_SO_MLcont_repr} alone. However, these two conditions coincide if the operator $A$ is self-adjoint ($\varphi_s=0$) and $\beta=1$.
\FloatBarrier
\begin{example}\label{ex:FCPML_ex1_hom_R_eigenfunction}
	\begin{figure}[tb]
		\newdimen\lentwosubfig
		\lentwosubfig=0.5\linewidth
		\hspace*{-0.5em}
		\ifpdf
			\begin{overpic}[width=0.98\lentwosubfig, viewport=100 40 1160 900, clip=true]%
				{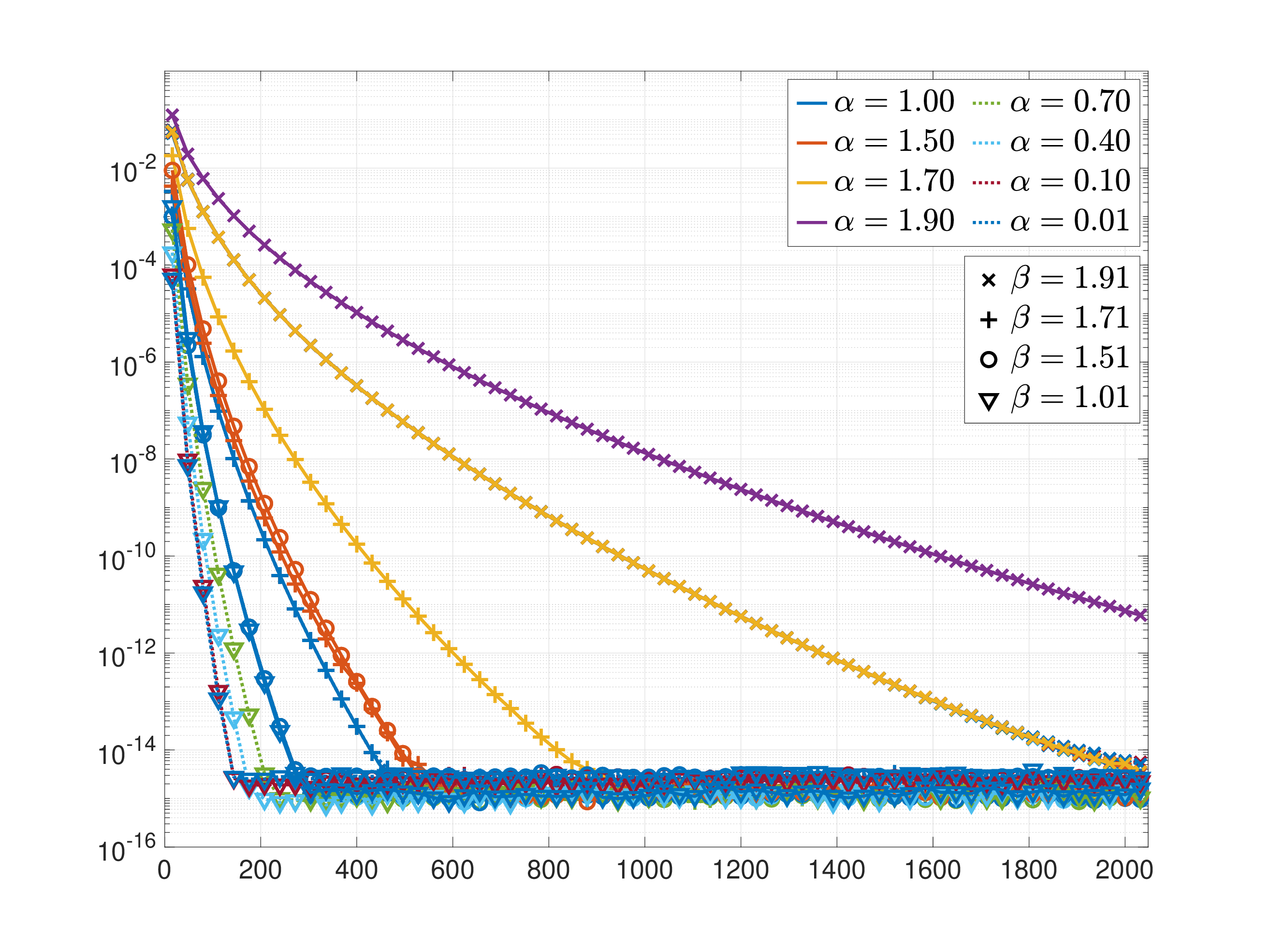}
				\put(48,-2){\small $N$}
				\put(-2,40){\small \rotatebox{90}{$\mathrm{err}_{\mathrm{h}}$}}
				\put(8,75){\scriptsize ({\bf a})}
			\end{overpic}
		\else
			\includegraphics[width=0.98\lentwosubfig, viewport=100 40 1160 900, clip=true]%
			{Ex1/FCPML_hom_error_vs_N_om_beta_1_1.5_1.7_1.9_nf0_1_ndf_4_D_1_large}
		\fi
		\hspace*{-0.5em}
		\ifpdf
			\begin{overpic}[width=0.98\lentwosubfig, viewport=100 40 1160 900, clip=true]%
				{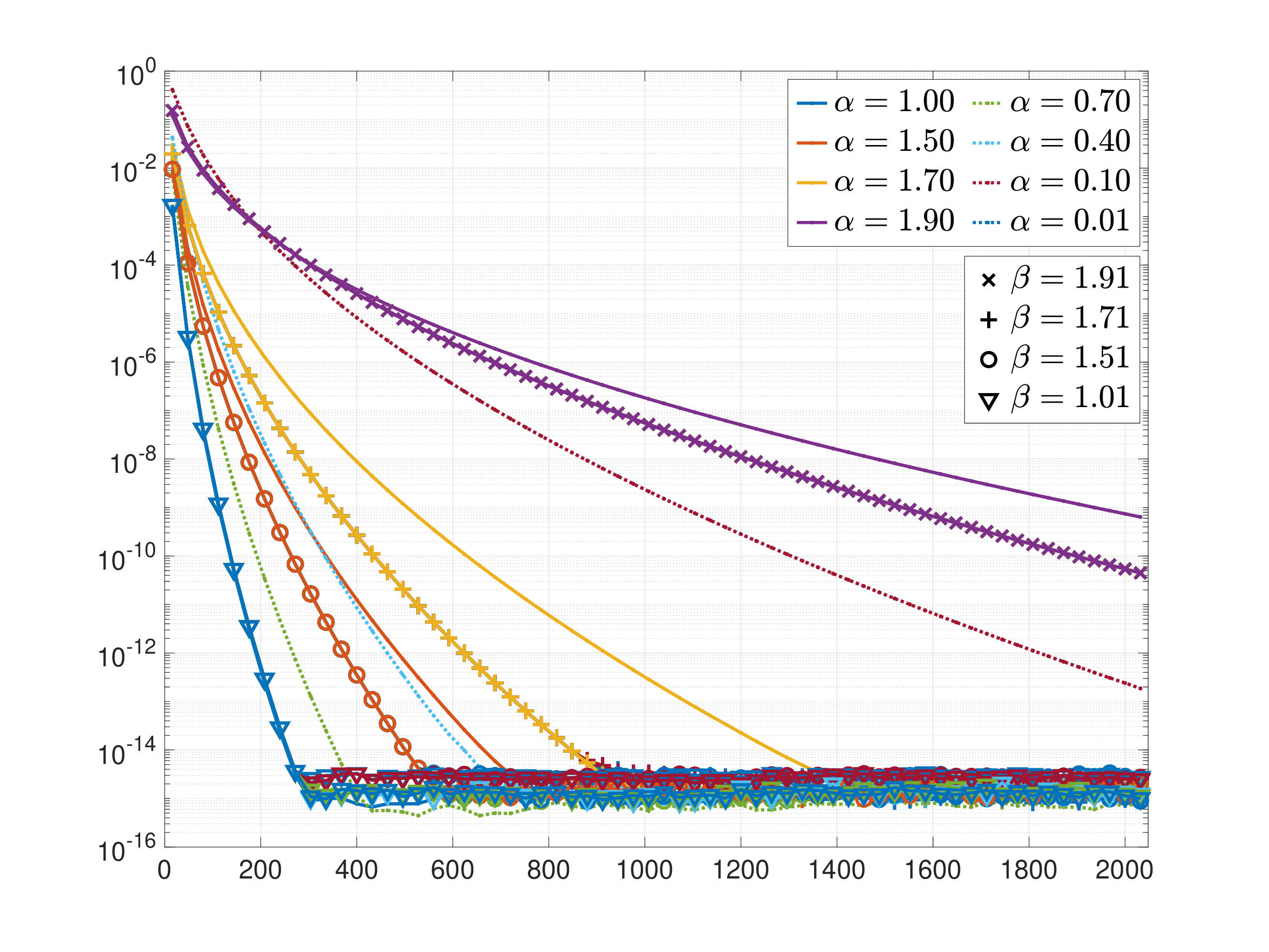}
				\put(48,-2){\small $N$}
				\put(-2,40){\small \rotatebox{90}{$\mathrm{err}_{\mathrm{h}}$}}
				\put(8,75){\scriptsize ({\bf b})}
			\end{overpic}
		\else
			\includegraphics[width=0.98\lentwosubfig, viewport=100 40 1160 900, clip=true]%
			{Ex1/FCPML_hom_error_vs_N_beta_1_1.5_1.7_1.9_nf0_1_ndf_4_D_1_large}
		\fi
		\caption[Example 1: Homogeneous problem error versus N]{%
			Sup-norm error of the approximate solution to fractional Cauchy problem \eqref{eq:FCP_DEBC},  \eqref{eq:FCP_ex1_hom_A}, \eqref{eq:FCP_ex1_hom_IV} with $\delta = 0$, $f(t) = 0$, $T=1$ and the angular size:
			(\textbf{a})  $\omega(\Theta) = \omega_c$; %
			(\textbf{b}) $\omega(\Theta) = \omega_\star$. %
			\vspace{-8pt}
		}
		\label{fig:FCPML_Ex1_hom_err_vs_N}
	\end{figure}
	Let us consider problem \eqref{eq:FCP_DEBC} with $A$ being defined as a negative Laplacian accompanied by the Dirichlet boundary conditions on $[0, 1]$:
	\begin{equation}\label{eq:FCP_ex1_hom_A}
		\begin{split}
			 & Au= - \frac{d^2}{dx^2}u, \quad \forall u \in D(A) \equiv \{u(x) \in H^2(0,1): \ u(0) = u(1) =  0\}.
		\end{split}
	\end{equation}
	For the given $k_0$, $k_1 \in \N$, $\delta \geq 0$, we choose $u_0, u_1$ as follows
	\begin{equation}\label{eq:FCP_ex1_hom_IV}
		u_0 = \left(x-x^2\right)^{2\delta}\sin{\pi k_0 x}, \quad u_1 = \sin{\pi k_1 x}.
	\end{equation}
	In the first set of experiments, we consider $\delta = 0$, $f(t) = 0$, $T=1$, $k_0 = 1$, $k_1 = 4$.
	Then, $u_0(x)$ and $u_1(x)$ are the eigenfunctions of $A$.
	In such case, fractional Cauchy problem \eqref{eq:FCP_DEBC},  \eqref{eq:FCP_ex1_hom_A}, \eqref{eq:FCP_ex1_hom_IV} admits the exact solution \cite[Sec. 1.3]{Bazhlekova2001}:
	\[
		u_\mathrm{h}(t) = E_{\alpha,1}(- 16 \pi^2 t^\alpha)  \sin{4\pi x} + H(\alpha-1)E_{\alpha,2}(-\pi^2 t^\alpha) \sin{\pi x}.
	\]
	Here, $H(\alpha)$ is the Heaviside function.
	The action of resolvent $R(z, A)$ on the chosen initial values can be evaluated explicitly: $R(z, A)\sin{\pi k x}
		= (z I -A)^{-1}\sin{\pi k x}
		= (z - \pi^2 k )\sin{\pi k x}$.
	Thus, the error of numerical solution $\wt{u}_\mathrm{h}^N(t)$ will not contain the spatial discretization component.
	To quantify the experimental error we define
	$\mathrm{err}_{\mathrm{h}}(N)
		=  \max\limits_{t \in \mathcal{T}} 	\left\|u_\mathrm{h}(t) - \wt{u}_{\mathrm{h}}^N(t)\right\|_\infty$.
	It is calculated using the MATLAB implementation\footnote{The code is available at \url{github.com/DmytroSytnyk/FCPML2025}} of the numerical scheme, evaluated at the uniform grid $\mathcal{T} \subset [0, T]$ including both endpoints.
	We conducted two series of experiments with $\kappa =1$, $\varphi_s = \pi/60$ and $\omega(\Theta) \in \{\omega_c, \omega_\star\}$.

	\Cref{fig:FCPML_Ex1_hom_err_vs_N}(a) illustrates the behavior of $\mathrm{err}_{\mathrm{h}}(N)$ for all feasible combinations of $\alpha \in \{0.1, 0.3, 0.5, 0.7, 1,  1.5, 1.7, 1.9\}$ and $\beta \in \{1.01,  1.51, 1.71, 1.91\}$ under the setting $\omega(\Theta) = \omega_c$, mentioned in  \cref{rem:FCPML_OPC_omega}.
	For each combination, the error decays exponentially, as predicted by \eqref{eq:FCP_hom_sol_err_est}, until it reaches the round-off plateau below $10^{-14}$ (not shown for $\alpha=1.9$).
	Moreover, for a fixed $\beta$, the convergence order $c=\sqrt{\pi \omega(\Theta)}$ increases as $\alpha \to 0$ until $\alpha/2 \leq \beta - 1 + \varphi_s/\pi $, and remains constant thereafter.
	In the plot, this  occurs at $\alpha = 1.7$ for $\beta=1.91$, at $\alpha = 1$ for $\beta = 1.71, 1.51$, and at $\alpha = 0.1$ for $\beta = 1.01$.
	We recall that the change of $\omega(\Theta)$ is causing the reposition of $\Gamma_I$ and the subsequent re-evaluation of all resolvents in \eqref{eq:FCP_SO_MLcor_sinc_quad}, \eqref{eq:FCP_ISO_exp_cor_repr_par}.

	Such costly re-evaluation can be avoided if we set the angular size to $\omega(\Theta) = \omega_\star$.
	Corresponding error behavior is depicted in \cref{fig:FCPML_Ex1_hom_err_vs_N}(b), where we additionally plotted the results of our earlier method \cite{Sytnyk2023}, calculated using the equivalent parameter setup (colored curves without markers).
	The setting $\omega(\Theta) = \omega_\star$ leads to a slightly lower  experimental convergence rate than $\omega(\Theta) = \omega_c$, but the rate remains fixed across different values of $\alpha < \beta = \mathrm{const}$.
	In both considered settings, the subordination based approximation $\wt{S}_{\alpha}^{N}(1.01, t)$,  $\alpha \in (0.1, 0.5]$, turns out to be $2$ -- $9$ times faster than our earlier method.
	Somewhat surprisingly, we also perceive an evident improvement of the method's convergence versus \cite{Sytnyk2023} in the case when $\alpha \approx \beta$ (cf. same color curves with and without markers in \cref{fig:FCPML_Ex1_hom_err_vs_N}(b)).
	In our opinion, it might be caused by the better accuracy of the OPC algorithm for $E_{\gamma,\sigma}(s)$, $s \in \Gamma_I$, compared to the MATLAB inbuilt evaluation routines for $e^s$.

	In order to test whether the results from \cref{fig:FCPML_Ex1_hom_err_vs_N} are generalizable to the more application-oriented problems, we conducted another series of experiments for \eqref{eq:FCP_DEBC},  \eqref{eq:FCP_ex1_hom_A} with $f(t) = 0$; $T=1$; and the initial condition $u(0) = u_0/\|u_0\|_{\infty}$, determined by \eqref{eq:FCP_ex1_hom_IV}, with $k_0=3$ and the values of $\delta$ from \cref{tab:Ex1_regularity_vs_accuracy}.
	This time, we used the second-order FD discretization $\wt{A}$ of \eqref{eq:FCP_ex1_hom_A} along with the explicit linear solver for evaluation of $(z I + \wt{A})^{-1} u_0$.
	The calculated data displays an impact of the $\delta$-dependent smoothness of $u_0$ on the value of $N$, needed to reach $\mathrm{err}_{\mathrm{h}} \leq 10^{-14}$.

	\begin{table}[tbhp]
		\footnotesize
		\caption{\noindent The impact of regularity parameter $\delta$ from \eqref{eq:FCP_ex1_hom_IV} on the experimentally estimated value of $N = N_e$ needed to achieve $\mathrm{err}_{\mathrm{h}} \leq 10^{-14}$, with $\alpha\leq 1$, $k_0=3$, $\beta = 1.01$ and $h$ determined by \cref{thm:FCPML_hom_sol_appr}.}%
		\label{tab:Ex1_regularity_vs_accuracy}
		\centering
		\begin{tabular}{|l | c c c c c c c c c c c|}
			\hline
			$\hspace{3.2em}\delta$    & $0.01$          & $0.1$           & $0.2$           & $0.3$           & $0.4$           & $0.5$           & $0.6$           & $1.0$           & $2.1$           & $3.9$           & $7.1$           \\
			\hline
			$N_e$ ($\kappa = \delta$) & ---             & 2109            & 1092            & 733             & 546             & 439             & 367             & 286             & 606             & 1138            & ---             \\
			$N_e$ ($\kappa = 0.7$)    & 313             & 290             & 310             & 318             & 317             & 317             & 317             & 316             & 316             & 316             & 316             \\
			$N_e$ ($\kappa = 0.8$)    & \underline{259} & \underline{274} & \underline{265} & \underline{251} & \underline{254} & 262             & 261             & 261             & \underline{259} & \underline{259} & \underline{259} \\
			$N_e$ ($\kappa = 0.9$)    & 291             & 308             & 298             & 283             & 265             & \underline{259} & \underline{259} & \underline{260} & 262             & 265             & 268             \\
			$N_e$ ($\kappa = 1$)      & 323             & 342             & 330             & 313             & 294             & 285             & 286             & 286             & 289             & 292             & 297             \\
			$N_e$ ($\kappa = 1.1$)    & 355             & 376             & 363             & 345             & 323             & 314             & 314             & 315             & 317             & 320             & 325             \\
			\hline
		\end{tabular}
	\end{table}
	Judging from the \cref{tab:Ex1_regularity_vs_accuracy}, we conclude that the proposed scheme is robust with respect to $\delta$.
	Moreover, a universally accepted strategy of setting $\kappa=1$ notwithstanding the actual spatial smoothness of $u_0$ is, in fact, practically viable.
	It is interesting, though, that $\kappa = 1$ is only quasi-optimal in terms of error.
	The optimal $\kappa \in (0, 1]$ can be estimated via the parabolic fit of the error $\left\|  u(0) - \wt{u}_{\mathrm{h}}^N(0)\right\|_\infty$, that is readily available for arbitrary $u_0$, $u_1$ and $A$.

	Next, we turn to the experimental validation of inhomogeneous solution scheme \eqref{eq:FCP_inhom_sol_appr}, \eqref{eq:FCP_inhom_sol_N}.
	Let $u_0 = u_1 = 0$, $f(t) = \sin{\pi x} + t\sin{4 \pi x}$, then the mild solution to problem \eqref{eq:FCP_DEBC},  \eqref{eq:FCP_ex1_hom_A} can be expressed as 
	\[
		u_\mathrm{ih}(t) =
		\left (1 - E_{\alpha,1}(- \pi^2 s^\alpha)\right )\tfrac{\sin{\pi x}}{\pi^2}
		+  \int_0^t \tfrac{E_{\alpha,1}(- 16 \pi^2 (t-s)^\alpha) } {\Gamma(\alpha+1)}s^\alpha ds \sin{ 4 \pi x}.
	\]
	We approximate the integral to the round-off precision $\leq 10^{-15}$ and consider the resulting $u_\mathrm{ih}(t)$  as a proxy for exact solution \eqref{eq:FCP_InhomSol_rep}.
	In all other respects, our next experiment copies the setup of the homogeneous solution experiment with $\omega(\Theta) = \omega_\star$.
	The behavior of the calculated error $\mathrm{err}_{\mathrm{ih}}(N) \equiv  \max\limits_{t \in \mathcal{T}} 	\left\|  u_\mathrm{ih}(t) - \wt{u}_{\mathrm{ih}}^N(t)\right\|_\infty$ is depicted in \cref{fig:FCPML_Ex1_inhom_err_vs_N}(a).
	In this plot, the exponentially decaying error curves for fixed $\beta$ and different $\alpha$ overlap, which indicates that the convergence order of $\wt{u}_{\mathrm{ih}}^N(t)$ is independent of $\alpha$, exactly as postulated by \cref{thm:FCP_ML_inhom_sol_appr}.

	\begin{figure}[tb]
		\hspace*{-6pt}
		\centering
		\newdimen\lentwosubfig
		\lentwosubfig=0.485\linewidth
		\hspace*{0.4em}
		\ifpdf
			\begin{overpic}[width=0.98\lentwosubfig, viewport=100 40 1160 895, clip=true]%
				{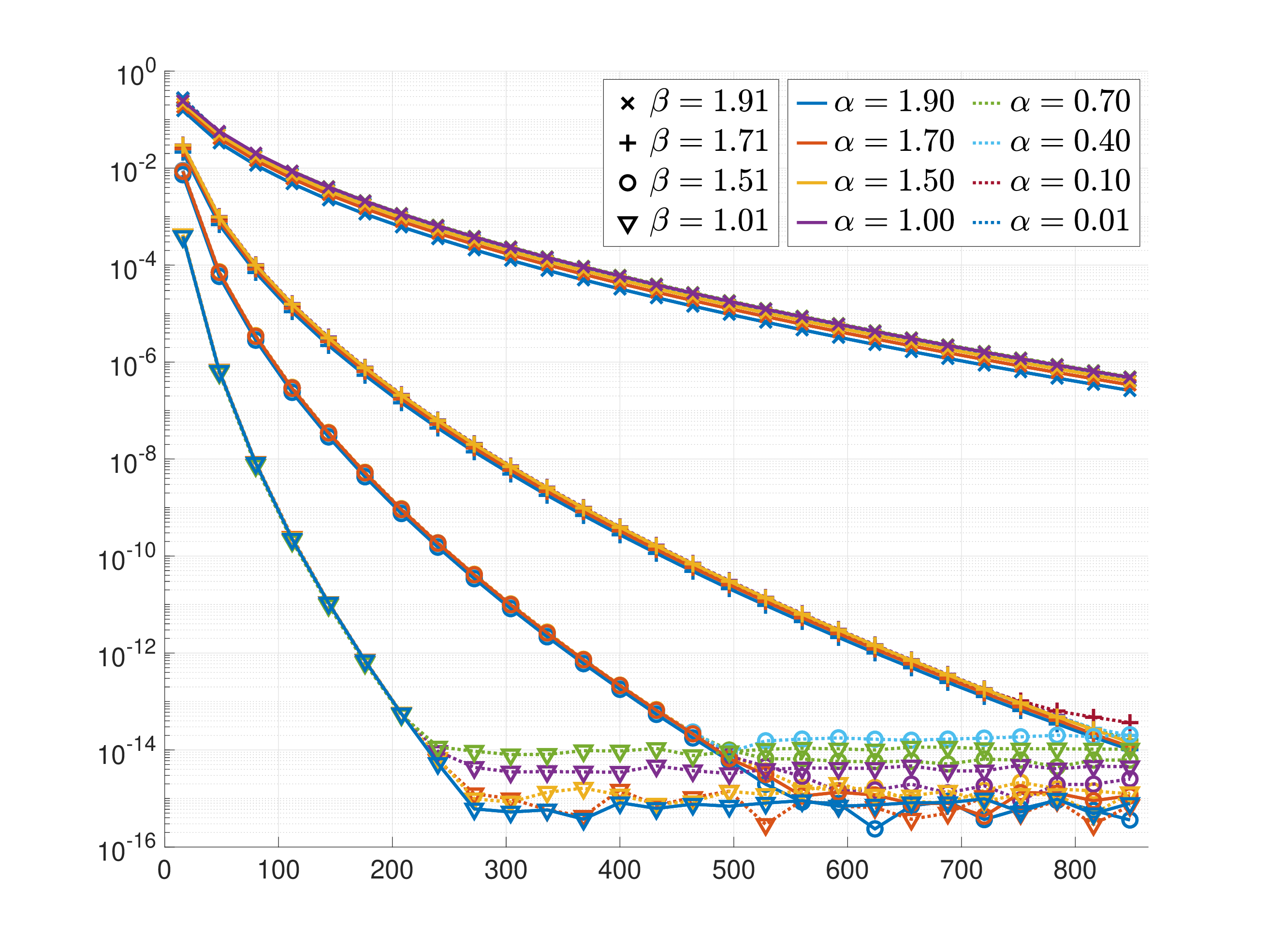}
				\put(48,-2){\small $N$}
				\put(-2,38){\small \rotatebox{90}{$\mathrm{err}_{\mathrm{ih}}$}}
				\put(10,75){\scriptsize ({\bf a})}
			\end{overpic}
		\else
			\includegraphics[width=0.98\lentwosubfig, viewport=100 40 1160 895, clip=true]%
			{Ex1/FCPML_inhom_error_vs_N_beta_1_1.5_1.7_1.9_nf0_1_ndf_4_D_1_large}
		\fi
		\ifpdf
			\begin{overpic}[width=0.98\lentwosubfig, viewport=100 40 1160 900, clip=true]%
				{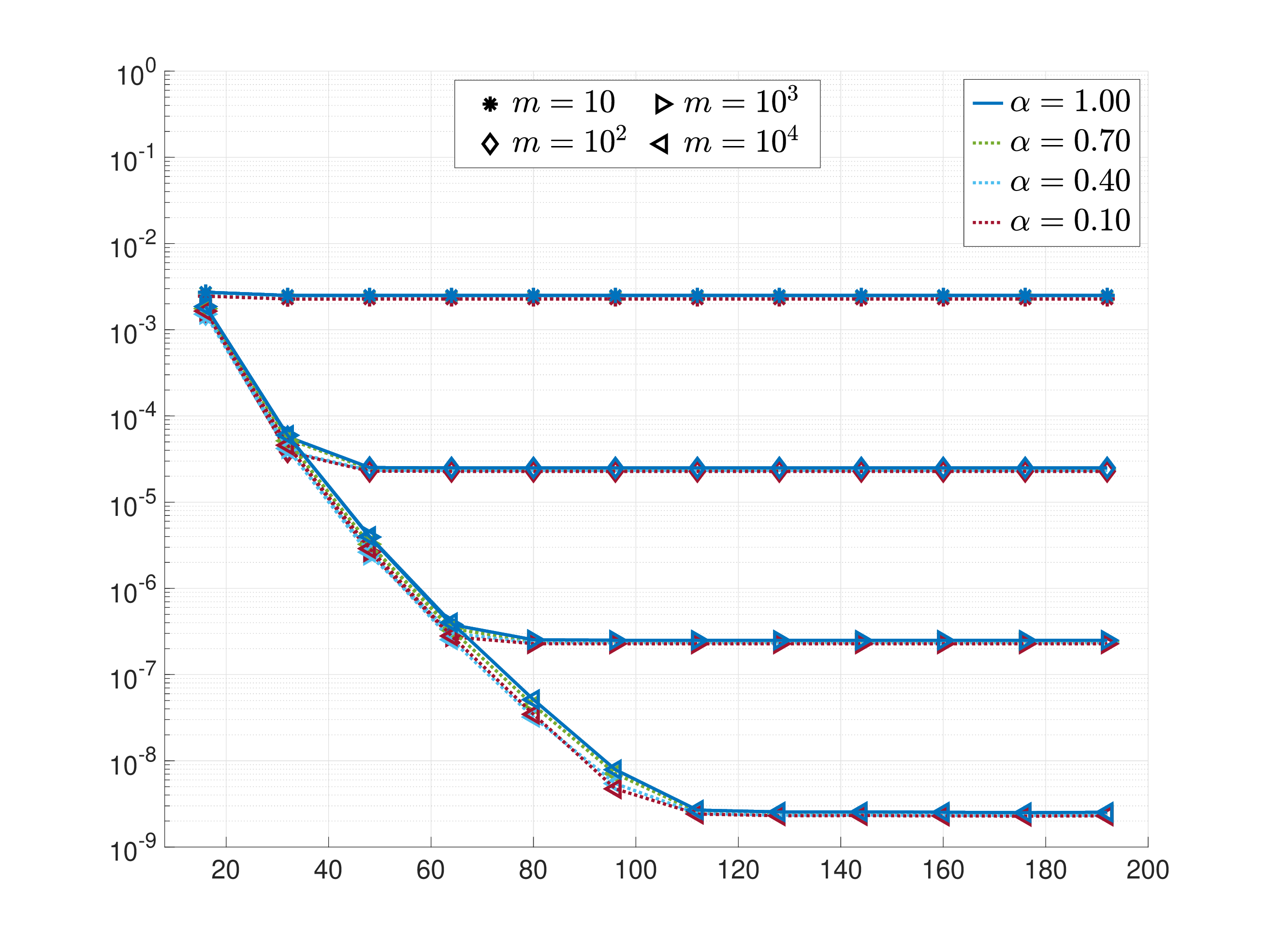}
				\put(48,-2){\small $N$}
				\put(-2,38){\small \rotatebox{90}{$\mathrm{err}$}}
				\put(10,75){\scriptsize ({\bf b})}
			\end{overpic}
		\else
			\includegraphics[width=0.98\lentwosubfig, viewport=100 40 1160 900, clip=true]%
			{Ex1/FCPML_Cauchy2_error_vs_N_beta_1_alpha_1.0_0.7_0.4_0.1_D1_large}
		\fi
		\caption[Example 1: Inhomogeneous problem error versus N]{%
			Sup-norm error of the approximate solution to fractional Cauchy problem \eqref{eq:FCP_DEBC},  \eqref{eq:FCP_ex1_hom_A}, with $T=1$, $\omega(\Theta) = \omega_\star$ and:
			(\textbf{a})  $u_0 = u_1 = 0$, $f(t) = \sin{\pi x} + t\sin{4 \pi x}$; %
			(\textbf{b})  $u_0$, $u_1$ and $f(t)$ calculated from the exact solution $u(t) = x^2 (x - 1) \left(x - t ^ 2 + \tfrac{1}{2} \right )$.%
			\vspace{-8pt}
		}
		\label{fig:FCPML_Ex1_inhom_err_vs_N}
	\end{figure}
	In the last series of experiments, we consider a general fractional Cauchy problem \eqref{eq:FCP_DEBC},  \eqref{eq:FCP_ex1_hom_A} 	with the data $u_0 = x^2 (x - 1) \left(x  + \tfrac{1}{2} \right )$, $f(0) = 1 + 3 x -12 x^2$, $f'(t) =   12 x t - 4 t - 2  \frac{x^2 (x - 1)}{\Gamma{(2 - \alpha)}} t^{1 - \alpha}$, calculated from the given exact solution $u(t) = x^2 (x - 1) \left(x - t ^ 2 + \tfrac{1}{2} \right )$.
	The approximate solution $\wt{u}^N(t) = \wt{u}_{\mathrm{h}}^N(t)  + \wt{u}_{\mathrm{ih}}^N(t)$  is evaluated using a fully-discretized method, comprised of the schemes for $\wt{u}_{\mathrm{h}}^N(t)$, $\wt{u}_{\mathrm{ih}}^N(t)$, in their $\alpha$  independent versions, and the spatial discretization $A \to \wt{A}$, discussed earlier.
	The resulting values of $\mathrm{err} \equiv  \max\limits_{t \in \mathcal{T}} 	\left\|  u(t) - \wt{u}^N(t)\right\|_\infty$ are shown in \cref{fig:FCPML_Ex1_inhom_err_vs_N}(b) for the spatial grid sizes $m = 10, 10^2, 10^3, 10^4$.
	In addition to the apparent method's stability for all considered combinations of $N$ and $m$, the displayed results give us a more concrete evidence supporting the use of static contour $\Gamma_I$ for propagator approximation \eqref{eq:FCP_SO_MLcor_sinc_quad}.
	While being around $5$ times slower than the methods with time-dependent contours \cite{lopez-fernandez1,Weideman2010}, our approximation $\wt{S}_{\alpha}^{N}(\beta, t)$ allows the resolvent reuse across different $t$, which quickly offsets the slower convergence during the computation of  $\wt{u}_{\mathrm{ih}}^N(t)$.
	For instance, the calculation of $\wt{J}_\alpha^{N_0} \wt{S}_{\alpha}^N(\beta, t)f(0)$ requires $2 N_0 + 1$ reevaluations of the propagator, hence our method's speed-up for this term is proportional to $N/{\min\left \{1, \alpha\right \}}$.
\end{example}

\FloatBarrier
\begin{example}\label{ex:FCPML_ex2_inv_alpha_hom_R_FD}
	In this example we consider an inverse problem of identifying the fractional order parameter $\alpha$ for  \eqref{eq:FCP_DEBC}, \eqref{eq:FCP_ex1_hom_A}, \eqref{eq:FCP_ex1_hom_IV}, with $f(t) = 0$, $\delta = 0$, $k_0 = 2$, $k_1 = 4$, from a set of pointwise solution measurements.
	This problem admits a unique solution \cite{Yamamoto2021}.
	We conducted 1000 independent experiments, where a uniformly randomly selected $\alpha \in [0.1, 1.6]$ was used to generate a "measured" data $u(t,x)$, $t \in \mathcal{T}$, $x = \pi/10$  from the exact solution.
	Then, two candidate values of $\alpha_{\mathrm{fit}}$ were computed via the minimization of $\sum\|u(t) - \wt{u}_{\mathrm{h}}^{128}(t)\|_\infty^2$.
	The approximate solution  $\wt{u}_{\mathrm{h}}^{128}(t)$ is obtained using: our earlier method \cite{Sytnyk2023}, which corresponds to $\beta = \alpha$ in \eqref{eq:FCP_SO_MLcor_sinc_quad};  the scheme from \cref{thm:FCPML_prop_appr}, with $\beta=1.6$, $\omega(\Theta) = \omega_\star$.
	The minimization process is implemented using MATLAB \verb*|lscurvefit| routine, with the initial guess $\alpha_0 = 0.85$ and the default stopping criteria.
	We used of the fact that $A$, $u_0$, $u_1$ are real-valued to reduce the number of evaluated terms in \eqref{eq:FCP_SO_MLcor_sinc_quad} from $2N+1$ to $N+1$ via Remark 1 of \cite{Sytnyk2023}.

	\begin{figure}[tb]
		\newdimen\lentwosubfig
		\lentwosubfig=0.5\linewidth
		\hspace*{0.1em}
		\ifpdf
			\begin{overpic}[width=0.98\lentwosubfig, viewport=90 60 1080 900, clip=true]%
				{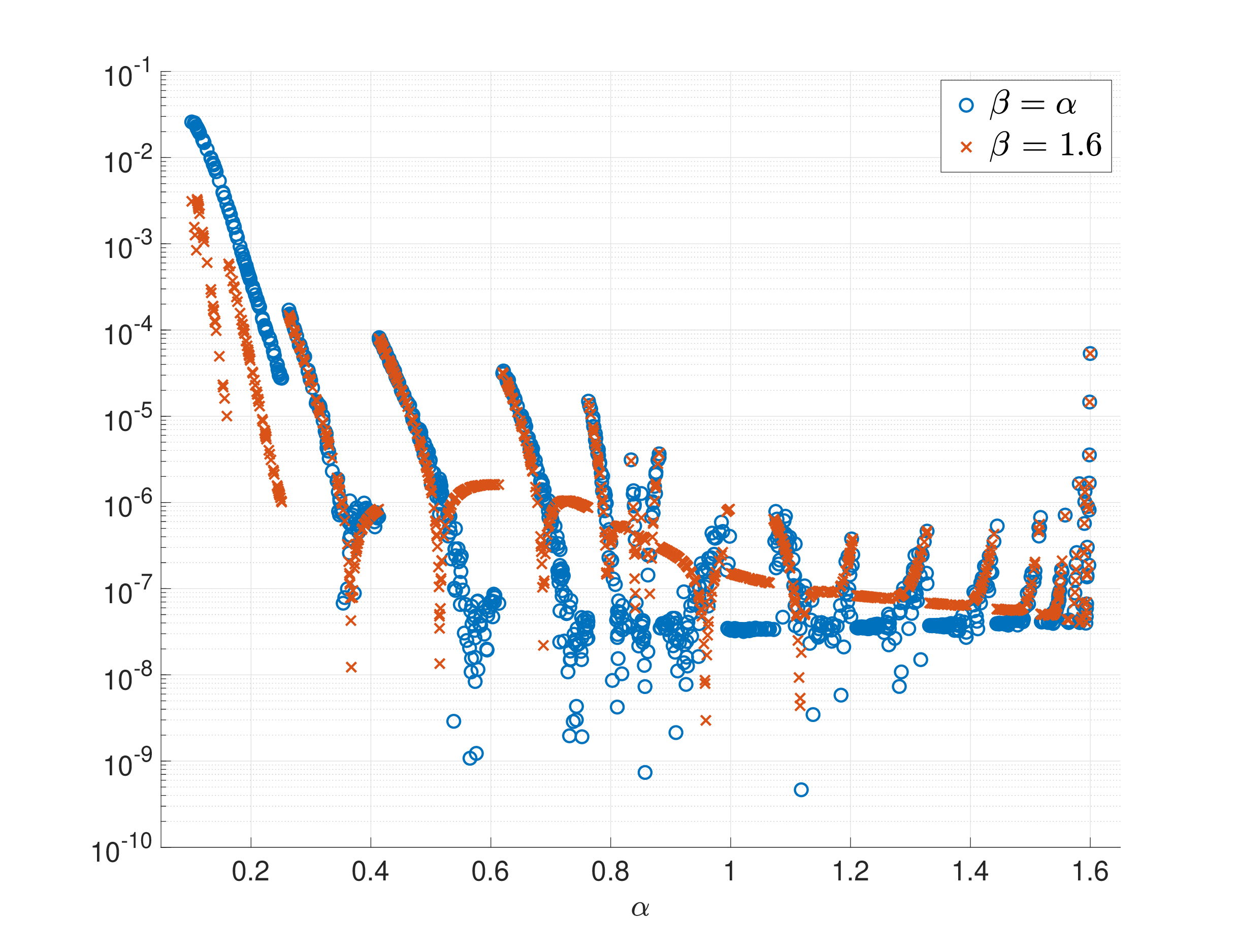}
				\put(42,-2){\small $\alpha$}
				\put(-2,33){\small \rotatebox{90}{$\mathrm{err}_{\alpha}$}}
				\put(10,74){\scriptsize ({\bf a})}
			\end{overpic}
		\else
			\includegraphics[width=0.98\lentwosubfig, viewport=90 60 1080 900, clip=true]%
			{Ex2/fit_err_vs_trials_FCPML_14_alpha_0.10_1.60_N_128_t_0_1_n0_2_n1_4_Nt_40_Nx_5000_trials_1000_color_inv}
		\fi
		\hspace*{-0.6em}
		\ifpdf
			\begin{overpic}[width=0.98\lentwosubfig, viewport=90 60 1080 900, clip=true]%
				{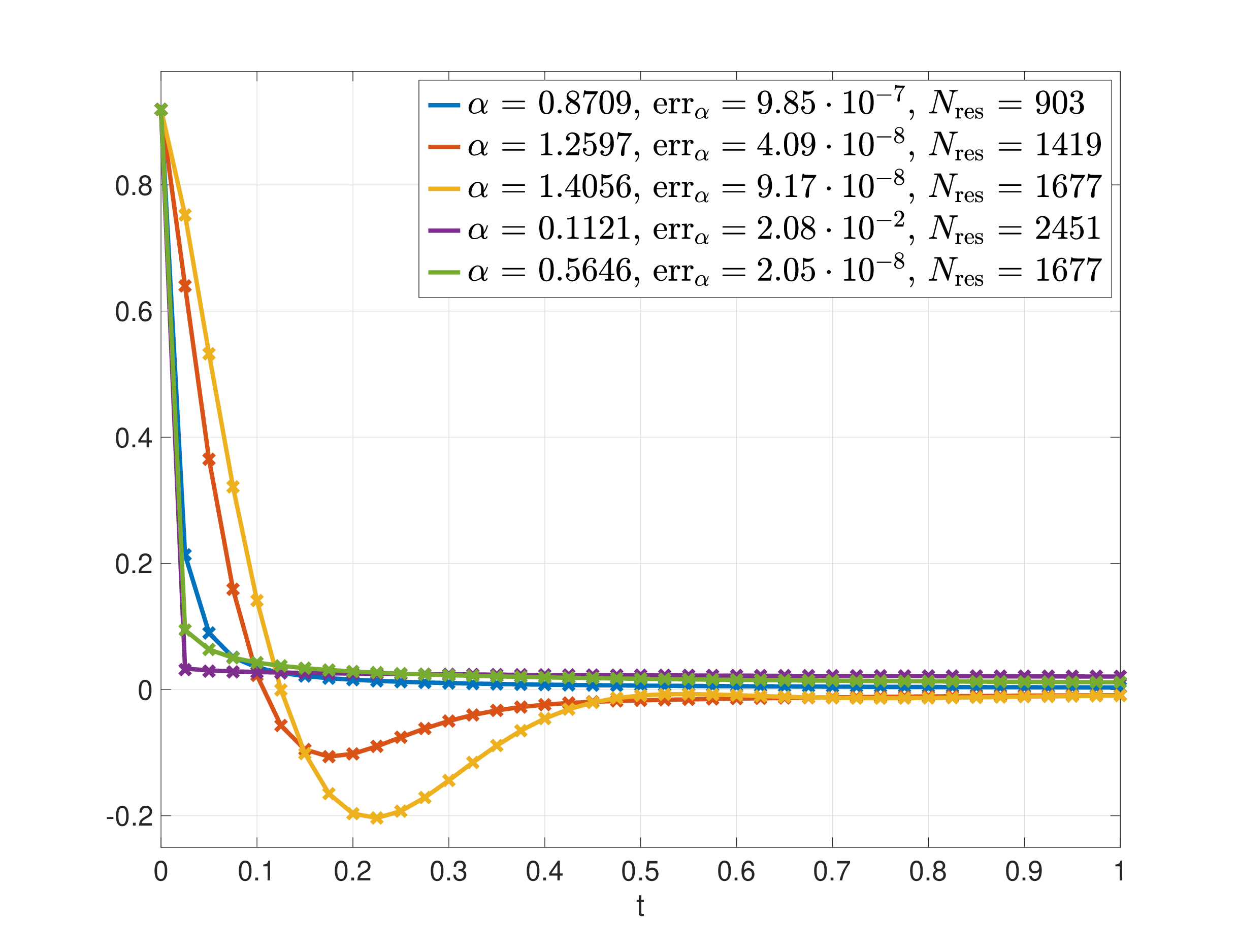}
				\put(46,-2){\small $t$}
				\put(0,35){\small \rotatebox{90}{$u$}}
				\put(53.95,50.5){\setlength{\fboxrule}{0.1pt}%
					\fcolorbox{black}{white}{\tiny \begin{minipage}{10.6em}
							\markercross \hspace*{0.05em} measured data\par
							\markerline \hspace*{0.02em} solution with fitted $\alpha$
						\end{minipage}}}
				\put(10,74){\scriptsize ({\bf b})}
			\end{overpic}
		\else
			\includegraphics[width=0.98\lentwosubfig, viewport=90 60 1080 900, clip=true]%
			{Ex2/fit_alpha_err_vs_t_FCPML_14_beta_alpha_alpha_0.10_1.60_N_128_t_0_1_n0_2_n1_4_Nt_40_Nx_5000_trials_5_si}
		\fi
		\caption[Example 2: Error of $\alpha$ fit]{%
			Result of the fractional order identification experiments.
			(\textbf{a}) Absolute fitting error $\mathrm{err}_{\alpha}$ for 1000 values of $\alpha$, uniformly randomly distributed on $[0.1, 1.6]$; %
			(\textbf{b}) Measured data along with the fitted approximate solution for $x = \pi/10$, $\mathcal{T} = \{0, 1/40, \ldots, 1\}$.
			\vspace{-8pt}
		}
		\label{fig:FCPML_Ex2_inv_prob_err_vs_alpha}
	\end{figure}
	The results are shown in \cref{fig:FCPML_Ex2_inv_prob_err_vs_alpha}.
	Here, we demonstrate the overall dependence of the fitting error $\mathrm{err}_{\alpha} \equiv |\alpha - \alpha_{\mathrm{fit}}|$ on $\alpha$, alongside with the individual experimental outcomes for five $\alpha$ samples, computed using the method from \cite{Sytnyk2023}.
	On average, it requires $N_{\mathrm{res}} \approx 1624$ unique evaluations of $(z^\beta I + \wt{A})^{-1}u_0$ per $\alpha_{\mathrm{fit}}$.
	Whereas, the subordination based scheme requires only $N_{\mathrm{res}} = 129$ unique resolvent evaluations $(z^\beta I + \wt{A})^{-1}u_0$ for all computed $\alpha_{\mathrm{fit}}$, which amounts to $\approx 12.5$ times improvement.
	One would expect even bigger gains in the presence of measurement error, because the reliable reconstruction of $\alpha_{\mathrm{fit}}$ from the noised data requires multiple identification experiments for regularization and uncertainty estimation \cite{Zhang2018}.
\end{example}

\section*{Concluding remarks}
In this work, we developed a new exponentially convergent numerical method for Caputo fractional propagator $S_\alpha(t)$ that combines contour based approximation technique from \cite{Sytnyk2023} and subordination identity \eqref{eq:FCP_SO_subord_repr}, for expressing $S_\alpha(t)$ in terms of the propagator $S_\beta(t)$, with order $\beta \in \left[\alpha, 2  (1 - {\varphi_s}/{\pi}) \right]$.
The error of the proposed method is characterized by bound \eqref{eq:FCP_hom_sol_err_est}, providing $\mathcal{O} (e^{-c\sqrt{\kappa\beta N}})$ convergence rate,  for any fixed $t \in [0,T]$.
After being incorporated into the numerical scheme for problem \eqref{eq:FCP_DEBC}, the new method helps us to overcome limitations of approximations \cite{McLean2010,Colbrook2022a,Sytnyk2023}, causing the degradation of convergence when $\alpha \to 0$.
Explicit error estimates for the numerical solution to \eqref{eq:FCP_DEBC} are provided by \cref{thm:FCPML_hom_sol_appr} and \cref{thm:FCP_ML_inhom_sol_appr}.
In addition to the aforementioned uniform convergence rate, these results prove that the scheme is capable of handling problem's data with limited spatial smoothness, i.e. $\exists \kappa,\chi > 0$: $u_0, f(0) \in D(A^\kappa)$, $f'(t) \in D(A^{\chi}) $.

Besides the parameters of \eqref{eq:FCP_DEBC} and $\beta$, the error of numerical solution also depends on the choice of the integration contour.
We offer two distinct contour selection strategies.
The first strategy can be used to achieve the maximal practically feasible convergence rate $\mathcal{O} (e^{-\pi\sqrt{(1- \varphi_s/\pi-\alpha/2 ) \kappa N}})$ by choosing  $\beta \in [\alpha, 1 - \varphi_s/\pi + \alpha/2 ]$, under assumption that $\alpha$ is fixed.
The second strategy is intended for situations when one needs to evaluate the numerical solution for several different values of $\alpha$.
It is able to reuse the calculated resolvents across the entire range of $\alpha \in (0, \beta]$, while maintaining a fixed convergence rate $\mathcal{O} (e^{-\sqrt{\pi \beta \omega_\star \kappa N}})$, $\omega_\star = \min{\{\pi, \tfrac{\pi - \varphi_s}{\beta}\}} -  \tfrac{\pi}{2} \max{\{1, \tfrac{1}{\beta}\}}$.
Comparison between the strategies is shown in \cref{fig:FCPML_Ex1_hom_err_vs_N}.
In \cref{ex:FCPML_ex1_hom_R_eigenfunction}, we perform systematic experimental validation of the theoretical estimates by inspecting both semi- and fully-discretized variants of the developed scheme for problem \eqref{eq:FCP_DEBC}.
The computational benefits of the second strategy are further analyzed in \cref{ex:FCPML_ex2_inv_alpha_hom_R_FD},
where we deal with identification of $\alpha$ in \eqref{eq:FCP_DEBC} from the set of pointwise solution measurements.
The proposed strategy leads to at least an order of magnitude reduction in the number of resolvent evaluations compared to the naive approach.

\ifpdf
{
	\footnotesize
	\bibliographystyle{siamplain}
	\bibliography{bibliography.bib}%
}
\else
\fi

\end{document}